\documentclass[11pt]{amsart}
\usepackage[margin=1in]{geometry}
\usepackage{amsmath,amssymb,amsthm,mathtools}
\usepackage{enumitem}
\usepackage{hyperref}
\usepackage{microtype}
\usepackage{xcolor}

\usepackage{esint}

\hypersetup{colorlinks=true, linkcolor=blue!60!black, citecolor=blue!60!black, urlcolor=blue!60!black}

\newtheorem{theorem}{Theorem}[section]
\newtheorem{lemma}[theorem]{Lemma}
\newtheorem{proposition}[theorem]{Proposition}
\newtheorem{corollary}[theorem]{Corollary}
\newtheorem{remark}[theorem]{Remark}
\numberwithin{equation}{section}

\newcommand{\R}{\mathbb R}
\newcommand{\dd}{\,d}
\newcommand{\diam}{\operatorname{diam}}
\newcommand{\osc}{\operatorname{osc}}
\newcommand{\dist}{\operatorname{dist}}

\newcommand{\eps}{\varepsilon}
\newcommand{\Om}{\Omega}

\newcommand{\pnu}{\partial_\nu}

\title[Serrin's Problem under Dirichlet Perturbations]{Serrin's Problem under Dirichlet Perturbations: Geometric Compactness and Sharp Planar Stability}

\author{Qinfeng Li}
\address{School of Mathematics, Hunan University, Changsha, P.R. China.}
\email{liqinfeng1989@gmail.com}

\author{Weihong Xie}
\address{Department of Mathematics, Central South University, Changsha, P.R. China.}
\email{wh.xie@csu.edu.cn}

\author{Hang Yang}
\address{School of Mathematics, Hunan University, Changsha, P.R. China.}
\email{hangyang0925@gmail.com}

\date{}

\begin{document}

	\begin{abstract}
		Let $\Omega\subset\mathbb R^n$ be a bounded connected Lipschitz domain and let $u_\Omega$ be the H\"older continuous representative of the solution of the compatible constant-flux Neumann problem
		\[
		-\Delta u_\Omega=1\quad\hbox{in }\Omega,
		\qquad
		\partial_\nu u_\Omega=-\frac{|\Omega|}{P(\Omega)}\quad\hbox{on }\partial\Omega,
		\qquad
		\int_{\partial\Omega}u_\Omega\,d\sigma=0.
		\]
		We measure the Dirichlet oscillation by
		\[
		O(\Omega):=\max_{\partial \Omega}u_\Omega-\min_{\partial \Omega} u_\Omega.
		\]
		In our earlier work \cite{LiXieYang2023}, we posed a stability question for Serrin's overdetermined problem under the Dirichlet oscillation and proved that the answer is negative in dimensions $n\ge3$. The present paper resolves the question in the planar convex class and develops a sharp quantitative theory without any a priori geometric nondegeneracy.		We show that convexity is essential for the stability results in dimension two by constructing fixed-area annuli with $O(\Omega_k)\to0$ that remain far from every disk. By contrast, if $\Omega_k\subset\mathbb R^2$ are convex, $|\Omega_k|=\pi$, and $O(\Omega_k)\to0$, then, up to translations, $\Omega_k$ converges in the Hausdorff topology to the unit disk. More quantitatively, there exist universal constants $C>1$ and $\delta_0>0$ such that any bounded convex domain $\Omega\subset\mathbb R^2$ with $|\Omega|=\pi$ and $O(\Omega)\le\delta_0$ satisfies the sharp linear estimate
		\[
		R_\Omega-r_\Omega+\inf_{z\in\mathbb R^2}d_H(\Omega,B_1(z))
		\le C\,O(\Omega),
		\]
		where $R_{\Omega}$ and $r_{\Omega}$ are the circumradius and inradius of $\Omega$, respectively. The linear order is shown to be optimal. No a priori diameter, curvature, or smoothness assumption is imposed. A central step is to prove that small boundary oscillation itself excludes long-thin degeneration and yields geometric compactness. The qualitative argument then uses the rough-domain Serrin rigidity theorem of Figalli--Zhang. The quantitative proof combines a boundary $P$-function argument yielding a new tangential-gradient estimate, a uniform linear boundary-growth estimate for the Neumann torsion function, the reverse-Serrin identity of Magnanini--Molinarolo--Poggesi, and a two-dimensional weighted oscillation estimate. We also study the strictly weaker one-sided averaged deficit, 
		\[
		A(\Omega):=\frac1{P(\Omega)}\int_{\partial\Omega}u_\Omega\,d\sigma-
		\min_{\partial\Omega}u_\Omega.
		\]
		In the planar convex class, $A(\Omega_k)\to0$ still forces convergence to a disk; quantitatively,
		\[
		R_\Omega-r_\Omega+\inf_{z\in \mathbb{R}^{2}} d_H(\Omega,B_1(z))
		\le C A(\Omega)^{2/3}
		\]
		for $|\Omega|=\pi$ and sufficiently small $A(\Omega)$.
	\end{abstract}
	
	\date{}
	
	\maketitle
	\tableofcontents

	\section{Introduction}
	Serrin's celebrated overdetermined problem asserts that if a bounded
	\(C^{2}\) domain \(\Omega \subset \mathbb{R}^{n}\) admits a solution to
	\begin{equation}\label{eq:Serrin}
		\begin{cases}
			-\Delta v = 1 & \text{in } \Omega,\\
			v = 0 & \text{on } \partial\Omega,\\
			\partial_{\nu} v = c & \text{on } \partial\Omega,
		\end{cases}
	\end{equation}
	where \(c\) is a constant, then \(\Omega\) must be a ball.
	Serrin's original proof relies on the method of moving
	planes~\cite{Serrin1971}. Shortly thereafter, Weinberger provided an
	alternative proof based on a \(P\)-function identity~\cite{Weinberger1971}.
	Since these seminal contributions, numerous alternative proofs and extensions
	to both linear and nonlinear operators have been developed. For a
	non-exhaustive selection of related works, we refer
	to~\cite{BNST08,CH98,PS89,GL89,FK08,FGK06,BK11,CS09}; see also the
	surveys~\cite{NT18,Schaefer01,FG08,FigalliZhang2025,MagnaniniSurvey2017} and the references
	therein.
	
	Extensions of Serrin's theorem to nonsmooth domains have been investigated
	since the 1990s. In 1992, Vogel~\cite{Vogel92} proved that if a \(C^{1}\)
	domain \(\Omega\) admits a solution to~\eqref{eq:Serrin}, then \(\Omega\)
	is necessarily of class \(C^{2}\), and hence Serrin's theorem applies.
	Subsequently, Prajapat~\cite{Prajapat} answered a question posed by
	Berestycki by establishing Serrin's conclusion when \(\Omega\) is of class
	\(C^{2}\) except possibly at a corner and the solution is \(C^{2}\) away
	from that corner. A major recent advance in this direction is due to Figalli
	and Zhang~\cite{FigalliZhang2025}. They proved that Serrin's theorem remains
	valid for sets of finite perimeter satisfying a uniform upper density bound;
	their framework also permits slit-type discontinuities. In particular, their
	result encompasses the case of Lipschitz domains, answering an open question posed in \cite{HuangLiLi2024}.  Alternative proofs for Serrin system on
	Lipschitz domains were given in~\cite{DongZhang2025} and \cite{DR}.

	The present paper continues a different stability problem, explicitly formulated as Question~1 in our earlier work~\cite{LiXieYang2023}. There the Neumann datum is kept constant and the Dirichlet trace is allowed to oscillate: we asked whether
	\[
	-\Delta u=1\quad\text{in }\Omega,
	\qquad
	\partial_\nu u=-\frac{|\Omega|}{P(\Omega)}\quad\text{on }\partial\Omega,
	\qquad
	\osc_{\partial\Omega}u:=\max_{\partial \Omega}u-\min_{\partial \Omega}u \to0,
	\]
	forces $\Omega$ to approach a ball. We proved in~\cite{LiXieYang2023} that the answer is negative in each dimension $n\ge3$, by means of fixed-volume rectangular examples, and left the two-dimensional case open. The present paper resolves that question in the planar convex class. We prove qualitative and sharp quantitative stability in this setting, show that convexity is essential, and treat a strictly weaker one-sided averaged boundary deficit. A detailed comparison with the reverse and classical Serrin stability theories is introduced after the statements of the main results.
	
	Let \(\Omega\subset\R^n\) be a bounded connected Lipschitz domain with perimeter \(P(\Omega)\).  We denote by \(u_\Omega\) the weak solution, unique up to constants, of
	\begin{equation}\label{eq:neumann-torsion}
		\begin{cases}
			-\Delta u_\Omega=1 & \text{in }\Omega,\\
			\partial_\nu u_\Omega=-\dfrac{|\Omega|}{P(\Omega)} & \text{on }\partial\Omega,
		\end{cases}
	\end{equation}
	normalized by
	\[
	\int_{\partial\Omega}u_\Omega\,d\sigma=0.
	\]
	The compatibility condition is automatic due to divergence theorem, and the weak formulation of \eqref{eq:neumann-torsion} is
	\begin{equation}\label{eq:weak-form}
		\int_\Omega \nabla u_\Omega\cdot\nabla\varphi\,dx
		=
		\int_\Omega \varphi\,dx
		-\frac{|\Omega|}{P(\Omega)}\int_{\partial\Omega}\varphi\,d\sigma
	\end{equation}
	for any $\varphi\in H^1(\Omega)$. Any weak solution to \eqref{eq:neumann-torsion} is often called \textit{the Neumann torsion function}.
	
	By the global H\"older regularity for this Neumann problem on Lipschitz domains, as obtained through the Kenig--Pipher argument \cite{KenigPipher1993} recalled in Remark~2.3 of Huang--Li--Li \cite{HuangLiLi2024}, \(u_\Omega\) admits a H\"older continuous representative on \(\overline\Omega\).  We always use this representative and define
	\begin{equation}\label{eq:deficit}
		O(\Omega):=\osc_{\partial\Omega}u_\Omega
		=\sup_{\partial\Omega}u_\Omega-\inf_{\partial\Omega}u_\Omega.
	\end{equation}
	This quantity is independent of the additive constant. In particular, if \(O(\Omega)=0\), then \(u_\Omega\) is constant on the boundary. After subtracting this constant, one obtains a solution of Serrin's overdetermined system. Thus \(O(\Omega)\) naturally measures the nonconstancy of the Dirichlet trace for Serrin's problem, and it is precisely the oscillation  considered in ~\cite[Question~1]{LiXieYang2023}.

	We first state the instability result. Part~\textup{(i)} below was proved in~\cite{LiXieYang2023} and is included for completeness; part~\textup{(ii)} is the new planar non-convex counterexample. Together they show that both dimension and convexity are essential.
	
	\begin{theorem}\label{thm:dimension-convexity}
		Let \(O(\Omega)\) be defined by \eqref{eq:deficit}.
		\begin{enumerate}[label=\textup{(\roman*)}]
			\item If \(n\ge3\), there exists a sequence of fixed-volume rectangles \(\Omega_k\subset\R^n\) such that
			\[
			O(\Omega_k)\to0,
			\]
			and, after any translations and rotations, the Hausdorff distance from \(\Omega_k\) to every ball of the same volume tends to infinity.
			
			\item If $n=2$, there exists a sequence of fixed-area annuli $A_{k}$ such that 
			\[O(A_{k})\to 0,\]
			and, after any translations and rotations, the Hausdorff distance from $A_{k}$ to every ball of the same volume tends to infinity.
		\end{enumerate}
	\end{theorem}
	
	On the other hand, we prove the following stability result.
	\begin{theorem}\label{thm:disk-conv}
		Let $\Omega_k\subset\R^2$ be bounded convex domains satisfying
		\[
		|\Omega_k|=\pi,
		\qquad
		O(\Omega_k)\to0.
		\]
		Then
		\[
		\inf_{z\in\R^2}d_H(\Omega_k,B_1(z))\to0.
		\]
		Equivalently, after suitable translations, $\Omega_k$ converges in the Hausdorff topology to the unit disk.
	\end{theorem}
	
	The main difficulty in Theorem~\ref{thm:disk-conv} is to derive geometric precompactness from the oscillation itself. After normalizing the diameter, a carefully chosen test function and boundary layer estimates yields a quantitative gap between its bulk and boundary averages on thin convex domains; see Lemma~\ref{lem:thin-lower} and Section~\ref{sec:diameter-compactness}. This rules out long-thin degeneration as in the example in Theorem \ref{thm:dimension-convexity} (i), and yields  a uniform diameter bound at fixed area.
	
	The diameter bound, together with the area normalization, ensures a uniform positive lower bound for the inradii. The resulting uniform Lipschitz geometry and the \(W^{1,p}\) theory of Geng--Shen~\cite{GengShen} provide compactness of suitably normalized Neumann torsion functions. The limit solves Serrin's overdetermined problem only in a weak form on a Lipschitz convex domain, so the rough-domain rigidity theorem of Figalli--Zhang~\cite{FigalliZhang2025} is precisely the result needed to identify it as a disk. To the best of our knowledge, this is the first use of their rough-domain theorem as the rigidity endpoint of a compactness argument for stability. See Section~\ref{sec:qualitative-stability} for details.
	
	Beyond this qualitative convergence, we obtain a sharp quantitative estimate without imposing any a priori nondegeneracy assumption. For a bounded convex domain $\Omega\subset\R^2$, let
	\[
	r_\Omega:=\sup\{r>0:\text{there exists }x\in\Omega\text{ with }B_r(x)\subset\Omega\}
	\]
	and
	\[
	R_\Omega:=\inf\{R>0:\text{there exists }x\in\R^2\text{ with }\Omega\subset B_R(x)\}
	\]
	denote its inradius and circumradius, respectively.
	
	\begin{theorem}\label{thm:intro-quantitative}
		There exist universal constants $C>1$ and $\delta_0>0$ such that any bounded convex domain $\Omega\subset\R^2$ with
		\[
		|\Omega|=\pi,
		\qquad
		O(\Omega)\le\delta_0,
		\]
		satisfies
		\begin{equation}\label{eq:intro-sharp-estimate}
			R_\Omega-r_\Omega
			+\inf_{z\in\R^2}d_H(\Omega,B_1(z))
			\le C\,O(\Omega).
		\end{equation}
		The linear order is optimal. More precisely, there exists a family of smooth,
		strictly convex, centrally symmetric domains $\Omega_t$, normalized by $|\Omega_t|=\pi$, such that
		\[
		O(\Omega_t)=\frac{|t|}{2}+o(|t|),
		\qquad
		R_{\Omega_t}-r_{\Omega_t}=2|t|+o(|t|).
		\]
	\end{theorem}
	
	Theorem~\ref{thm:intro-quantitative} is a full-convex-class estimate: its constants do not depend on an a priori diameter bound, curvature bounds, touching-ball radii, or boundary smoothness. After Theorem~\ref{thm:disk-conv} supplies uniform diameter and inradius bounds, Section~\ref{sec:quantitative-full-convex} combines a tangential-gradient estimate for the Neumann torsion function, a boundary $P$-function estimate, and the reverse-Serrin identity of Magnanini-Molinarolo-Poggesi \cite{MagnaniniMolinaroloPoggesi2024} to obtain a preliminary quantitative bound. Section~\ref{sec:sharp-stability} adds a uniform linear boundary-growth estimate and a control of the normal derivative of the harmonic remainder by the radial defect. Together with a two-dimensional weighted oscillation estimate, these ingredients yield \eqref{eq:intro-sharp-estimate}. The perturbative family above proves that the linear order is sharp.

	The boundary oscillation   above captures the full range of the boundary
	trace. It is natural to ask whether a substantially weaker, one-sided
	$L^1$-type deficit still detects balls.
	We therefore set
	\begin{equation}\label{eq:intro-A-deficit}
		A(\Omega):=\frac1{P(\Omega)}\int_{\partial\Omega}u_\Omega\,d\sigma
		-\min_{\partial\Omega}u_\Omega,
	\end{equation}
	and show that the similar qualitative picture survives.   
	
	\begin{theorem}\label{thm:intro-average}
		Let $\Omega_k\subset\R^2$ be bounded convex domains with $|\Omega_k|=\pi$ and $A(\Omega_k)\to0$. Then
		\[
		\inf_{z\in\R^2}d_H(\Omega_k,B_1(z))\to0.
		\]
		Moreover, there exist universal constants $C>1$ and $\delta_A>0$ such that, if $\Omega\subset\R^2$ is a bounded convex domain  satisfying  $|\Omega|=\pi$ and $A(\Omega)\le\delta_A$, then
		\begin{align}
			\label{eq:2/3stability}
			R_\Omega-r_\Omega
			+\inf_{z\in\R^2}d_H(\Omega,B_1(z))
			\le C A(\Omega)^{2/3}.
		\end{align}
	\end{theorem}
	In view of $A(\Om)\le O(\Om)$, the high-dimensional rectangular and planar non-convex counterexamples persist for $A$.
	Thus the dimensional and convexity restrictions in
	Theorem~\ref{thm:intro-average} are essential in the same sense as for $O(\Omega)$.
	
	None of the preceding stability results follows formally from smallness of $A$. Nevertheless, 
	The proof  of  Theorem~\ref{thm:intro-average}  introduce two  independent ingredients. A quartic test function $x^2(1-x)^2$, together with the sharp Hermite--Hadamard inequality for nonnegative subharmonic functions, yields a lower bound $A(K)\ge c|K|$ for unit-diameter thin convex sets. Once diameter and inradius are controlled, the boundary $P$-function guarantees $O(\Omega)\le C A(\Omega)^{2/3}$, and Theorems~\ref{thm:disk-conv} and~\ref{thm:intro-quantitative} apply.  For more details, see Section \ref{sec:averaged-deficit}.

	Another motivation for the averaged deficit comes from optimal thermal insulation. The asymptotic optimization of thin insulating layers goes back at least to Buttazzo~\cite{Buttazzo1988}, and the two variational models relevant here were developed by Bucur, Buttazzo, and Nitsch~\cite{BucurButtazzoNitsch2017,BucurButtazzoNitschNotices2017}. For the energy problem, Huang--Li--Li~\cite{HuangLiLi2024} identified, for every nonround domain $\Omega$, a positive concentration-breaking threshold  established by
	\begin{equation}\label{eq:m0-intro}
		m_0(\Omega)=\frac{P(\Omega)^2}{|\Omega|}\,A(\Omega).
	\end{equation}
	Thus $A(\Omega)$ is precisely the geometric boundary deficit entering the critical mass. In the planar convex class, Theorem~\ref{thm:intro-average} therefore provides a stability result for this threshold. Indeed, when $|\Omega|=\pi$, the isoperimetric inequality yields $A(\Omega)\le m_0(\Omega)/(4\pi)$; consequently, $m_0(\Omega_k)\to0$ forces, up to translations, $\Omega_k\to B_1$ in the Hausdorff topology. Moreover, for sufficiently small $m_0(\Omega)$,
	\[
	R_\Omega-r_\Omega+\inf_{z\in\mathbb R^2}d_H(\Omega,B_1(z))
	\le C m_0(\Omega)^{2/3}.
	\]
	
	\medskip
	
	\noindent
	\textbf{Comparison with Neumann-data stability}
	
	It is important to distinguish the present problem from the classical quantitative theory for Serrin's overdetermined problem. In the classical setting, one prescribes the Dirichlet condition
	\[
	v=0\qquad\text{on }\partial\Omega
	\]
	and measures the deviation of the Neumann datum $\partial_\nu v$ from a constant. Quantitative symmetry results in this direction have been obtained by moving-plane and parallel-surface arguments~\cite{AftalionBuscaReichel1999,CiraoloMagnanini2014,CiraoloMagnaniniVespri2016,CiraoloMagnaniniVespri2018}, integral identities~\cite{BNST2008,MP201,MP202}, variational methods~\cite{Feldman2018}, and local implicit-function techniques~\cite{GilsbachOnodera2021}. For radius or Hausdorff conclusions, however, the available global estimates involve geometric parameters that exclude degeneration, whereas the optimal linear theory is local near the ball.
	
	This geometric nondegeneracy cannot, in general, be recovered from a small absolute Neumann-data discrepancy, even among fixed-area planar convex domains. Indeed, consider
	\[
	E_\varepsilon
	=
	\left\{(x,y):\frac{x^2}{a_\varepsilon^2}
	+\frac{y^2}{b_\varepsilon^2}<1\right\},
	\qquad
	a_\varepsilon=\varepsilon^{-1},\quad b_\varepsilon=\varepsilon,
	\]
	so that $|E_\varepsilon|=\pi$. Its Dirichlet torsion function is
	\[
	v_\varepsilon(x,y)
	=
	\frac{1-x^2/a_\varepsilon^2-y^2/b_\varepsilon^2}
	{2(a_\varepsilon^{-2}+b_\varepsilon^{-2})}.
	\]
	On $\partial E_\varepsilon$, parametrized by
	$x=a_\varepsilon\cos\theta$ and $y=b_\varepsilon\sin\theta$,
	\[
	-\partial_\nu v_\varepsilon
	=
	\frac{\sqrt{a_\varepsilon^{-2}\cos^2\theta
			+b_\varepsilon^{-2}\sin^2\theta}}
	{a_\varepsilon^{-2}+b_\varepsilon^{-2}}
	\le C\varepsilon.
	\]
	Since the compatible constant
	$c_\varepsilon=-|E_\varepsilon|/P(E_\varepsilon)$ also satisfies
	$|c_\varepsilon|\le C\varepsilon$, one has
	\[
	\|\partial_\nu v_\varepsilon-c_\varepsilon\|_{L^\infty(\partial E_\varepsilon)}
	\le C\varepsilon\to0,
	\]
	while $\diam(E_\varepsilon)\to\infty$. Thus a small absolute deviation of the Neumann datum from a constant does not by itself produce geometric compactness.
	
	By contrast, we impose the compatible constant Neumann datum exactly and measure
	\[
	O(\Omega)=\osc_{\partial\Omega}u_\Omega.
	\]
	In the planar convex class, Theorem~\ref{thm:disk-conv} shows that the smallness of boundary oscillation itself controls the diameter and hence, at fixed area, the inradius. Thus the Dirichlet-trace  oscillation  supplies the geometric compactness that must be imposed externally in much of the classical Neumann-data stability theory. The distinction is therefore not a formal interchange of boundary data: the two quantities provide  genuinely different coercive effects on the domain.
	
	\medskip
	
	\noindent
	\textbf{Relation with the reverse Serrin problem}
	
	The oscillation-only Dirichlet-perturbation problem was explicitly formulated as Question~1 in our 2023 work~\cite{LiXieYang2023}, where the high-dimensional rectangular counterexample was also established. Subsequently and independently, Magnanini--Molinarolo--Poggesi~\cite{MagnaniniMolinaroloPoggesi2024} studied the same constant-Neumann system under the terminology \emph{reverse Serrin problem}. Their general integral identity is an important ingredient in our quantitative proof. They obtained, for $C^2$ or $C^{2,\alpha}$ domains, estimates of the schematic form
	\[
	\label{MMPgen}
	\rho_e-\rho_i
	\le C\,\psi\!\left(\osc_{\partial\Omega}u+
	\|\nabla_{\partial\Omega}u\|_{L^p(\partial\Omega)}\right),
	\qquad p=2\ \text{or}\ \infty,
	\]
	where $\psi$ is a given continuous function, with in particular $\phi(t)=t$ for $n=2$, and the constants depend on the diameter and on uniform interior and exterior sphere radii, or on comparable boundary-regularity data. Thus those estimates involve both a stronger analytic control and an a priori nondegenerate geometric class.
	
The present paper completes the program initiated in~\cite{LiXieYang2023} in dimension two. Unlike the estimates above, our result assumes only the oscillation $O(\Omega)$ and no a priori nondegenerate geometric class: the smallness of $O(\Omega)$ itself first generates the uniform bounded diameter, while a boundary $P$-function estimate controls the tangential gradient. The reverse-Serrin identity of Magnanini--Molinarolo--Poggesi remains a very important ingredient, but it is combined with further linear-growth and harmonic-remainder estimates to close the argument at the sharp linear scale.

	\section{High-dimensional and non-convex instability}

	The following high-dimensional construction was already  presented in~\cite[Section~5]{LiXieYang2023}. We include the short proof for completeness and to keep the dimensional picture self-contained.
	
	\begin{proposition}[Rectangular counterexample]\label{prop:rectangles}
		Let \(n\ge3\).  There exists a sequence of fixed-volume rectangles \(\Omega_\eps\subset\R^n\) such that \(O(\Omega_\eps)\to0\) as \(\eps\downarrow0\), while the Hausdorff distance from any translate or rotation of \(\Omega_\eps\) to every ball of the same volume tends to infinity.
	\end{proposition}
	
	\begin{proof}
		Let
		\[
		\Omega=\prod_{i=1}^n(-a_i,a_i).
		\]
		Then, up to an additive constant, a direct calculation yields  that 
		\begin{equation}\label{eq:rectangle-u}
			u(x)=
			-\left(\sum_{i=1}^n\frac{x_i^2}{a_i}\right)
			\left(\sum_{i=1}^n\frac{2}{a_i}\right)^{-1}.
		\end{equation}
		Indeed, \(-\Delta u=1\), and on the face \(x_i=\pm a_i\),
		\[
		\partial_\nu u
		= -\frac{2a_i/a_i}{\sum_{j=1}^n2/a_j}
		= -\left(\sum_{j=1}^n\frac1{a_j}\right)^{-1},
		\]
		which is independent of the face.  This constant equals \(-|\Omega|/P(\Omega)\). If \(\sigma_m\) denotes the \(m\)-th elementary symmetric polynomial of \((a_1,\ldots,a_n)\), then
		\[
		\left(\sum_{i=1}^n\frac1{a_i}\right)^{-1}=\frac{\sigma_n}{\sigma_{n-1}}.
		\]
		Therefore
		\[
		\osc_{\partial \Omega}u
		\le
		\frac{\sigma_n}{2\sigma_{n-1}}
		\sum_{i=1}^n a_i
		=
		\frac{\sigma_1\sigma_n}{2\sigma_{n-1}}.
		\]
		Let $\Omega_\epsilon$ be the rectangle by choosing 
		\[
		a_1=\eps^{n-1},
		\qquad
		a_i=\frac1\eps\quad (i=2,\ldots,n).
		\]
		Then the volume is fixed.  
		Here \(\sigma_n=1\),
		\[
		\sigma_1\sim \frac{n-1}{\eps},
		\qquad
		\sigma_{n-1}\sim \frac1{\eps^{n-1}},
		\]
		and hence
		\[
		O(\Omega_\eps)\le \osc_\Omega u\le C\eps^{n-2}\to0.
		\]
		Moreover \(\operatorname{diam}(\Omega_\eps)\to\infty\), whereas balls of the same fixed volume have a uniformly bounded diameter. Hence the Hausdorff distance from any translate or rotation of \(\Omega_\eps\)  and any such ball tends to infinity.  This proves the claim and so establishes Theorem~\ref{thm:dimension-convexity}\textup{(i)}.  
	\end{proof}
	
	We next prove the failure of stability even in dimension $n=2$ for non-convex domains.
	
	\begin{proposition}
		Let
		\[
		A_R:=B_R\setminus \overline{B_{\rho(R)}},
		\qquad
		\rho(R)=\sqrt{R^2-1},
		\qquad R>1,
		\]
		be the family of planar concentric annuli of area $\pi$. Then
		\[
		O(A_R)\to 0
		\qquad\text{as }R\to\infty,
		\]
		while the Hausdorff distance between  any translate or rotation of $A_R$ to every ball of the same area tends to infinity.
	\end{proposition}
	
	\begin{proof}
		It is easy to verify that 
		\[
		u_R(x)
		=
		\frac{\rho(R)^2-|x|^2}{4}
		+
		\frac{\rho(R)R}{2}\log\frac{|x|}{\rho(R)}.
		\]
		Since $u_R$ is radial, it is constant on each connected component of
		$\partial A_R$. On the inner boundary $\partial B_{\rho(R)}$, $u_R$ vanishes. On the outer boundary $\partial B_R$, we have
		\[
		u_R(x)
		=
		\frac{\rho(R)^2-R^2}{4}
		+
		\frac{\rho(R)R}{2}\log\frac{R}{\rho(R)}.
		\]
		It follows from the definition of  $\rho(R)$ that   
		\[
		u_R\big|_{\partial B_R}
		=
		-\frac14
		+
		\frac{R\sqrt{R^2-1}}{2}
		\log\frac{R}{\sqrt{R^2-1}}, \qquad |x|=R.
		\]
		Then one calculates
		\[
		O(A_R)
		=
		\frac14
		-
		\frac{R\sqrt{R^2-1}}{2}
		\log\frac{R}{\sqrt{R^2-1}},
		\]
		and  
		\begin{equation*}
			O(A_R)\to0\qquad \text { as } \qquad\diam(\Omega_j)\to\infty.
		\end{equation*}
		Moreover, \(\operatorname{diam}(A_{R})\to\infty\). 
		Therefore, Theorem~\ref{thm:dimension-convexity}\textup{(ii)}  follows from the same argument as in part \textup{(i)}.
	\end{proof}

	\section{Small boundary oscillation excludes unbounded diameter at fixed area}
	\label{sec:diameter-compactness}
	The goal of this section is to prove that, under the area normalization $|\Omega|=\pi$, the condition $O(\Omega_k)\to0$ prevents the diameters from diverging. 
	Essentially, we shall prove the following lower bound estimate for boundary oscillation of Neumann torsion function.
	
	\begin{lemma}\label{lem:thin-lower}
		Let $K\subset\R^2$ be a bounded convex domain with $\diam(K)=1$, and let $v$ solve
		\[
		\begin{cases}
			-\Delta v=1 & \text{in }K,\\
			\partial_\nu v=-\dfrac{|K|}{P(K)} & \text{on }\partial K.
		\end{cases}
		\]
		Then there exist constants $a_0>0$ and $c_0>0$ such that
		\[
		\osc_{\partial K}v\ge c_0a
		\qquad\text{whenever }|K|=a\le a_0.
		\]
	\end{lemma}
	
	The following geometric result is crucial to obtain  Lemma \ref{lem:thin-lower}.
	
	\begin{proposition}
		\label{prop:gap}
		Let $K$ be a convex domain in $\mathbb{R}^2$ with $\diam(K)=1$. Suppose that, after a rigid motion, $(0,0)$ and $(1,0)$ are two endpoints of a diameter of $K$. Let $\psi(x):=x(1-x)$. Then there exist universal constants $a_0>0$ and $c_1>0$ such that, for every such $K$ with $|K|\le a_0$, 
		\begin{align}
			\label{eq:moment-gap}
			\frac{1}{|K|}\int_K\psi\,dx\,dy
			-\frac{1}{P(K)}\int_{\partial K}\psi\,d\sigma
			\ge c_1|K|.
		\end{align}
	\end{proposition}

	Let $K$ be as in Proposition~\ref{prop:gap}.  Then we can write
	\[
	K=\{(x,y):0<x<1,\ -q(x)<y<p(x)\},
	\]
	where $p,q\ge0$ are concave and vanish at $0$ and $1$. By standard approximation and convergence results, it is enough to prove Proposition~\ref{prop:gap} for $p,q \in C^2(0,1)\cap C[0,1]$. Throughout this section, we assume such regularity for $p, q$.
	
	\begin{lemma}\label{lem:tent}
		Let $K$ and $\psi$ be as above. Then there exists a nonnegative Radon measure $\beta$ on $(0,1)$ such that 
		\begin{equation}\label{eq:beta-mass}
			\beta([0,1])=2,
		\end{equation}
		and
		\begin{align}
			\label{eq:psi-H-identity}
			\frac{1}{|K|}\int_K \psi \, dxdy=\frac{1}{6}+\frac{1}{12}\int_0^1 t(1-t) d\beta(t).
		\end{align}
		In particular,
		\begin{equation}\label{eq:moment-ineq}
			\frac{1}{|K|}\int_K \psi \, dxdy\ge \frac{1}{6}.
		\end{equation}
	\end{lemma}
	
	\begin{proof}
		Let $h=p+q$ and $|K|=a$,  then 
		\[
		h(0)=h(1)=0,
		\qquad
		\int_0^1h(x)\,dx=a>0.
		\]
		Set $H(x)=h(x)/a$.  We claim that there exists a nonnegative Radon measure $\beta$ on $(0,1)$ such that
		\begin{equation}\label{eq:tent-representation}
			H(x)=\int_0^1T_t(x)\,d\beta(t),
		\end{equation}
		where
		\[
		T_t(x)=
		\begin{cases}
			\dfrac{x}{t},&0\le x\le t,\\[2mm]
			\dfrac{1-x}{1-t},&t\le x\le1.
		\end{cases}
		\]	
		Indeed, recall that the Green function for the operator $-\tfrac{d^2}{dx^2}$ on $(0,1)$  satisfies 
		\[
		G(x,t)=
		\begin{cases}
			x(1-t),&x\le t,\\
			t(1-x),&x> t.
		\end{cases}
		\]
		Then
		\[
		H(x)=-\int_0^1G(x,t)H''(t)\,dt.
		\]
		Since
		\[
		G(x,t)=t(1-t)T_t(x),
		\]
		we obtain \eqref{eq:tent-representation} by setting
		\[
		d\beta(t)=-t(1-t)H''(t)\,dt,
		\]
		which is a nonnegative Radon measure.
		
		By \eqref{eq:tent-representation} and the identity
		\[
		\int_0^1T_t(x)\,dx=\frac12,
		\]
		Fubini's theorem   provides 
		\[
		1=\int_0^1H(x)\,dx
		=\int_0^1\left(\int_0^1T_t(x)\,dx\right)d\beta(t)
		=\frac12\beta([0,1]),
		\]
		so \eqref{eq:beta-mass} follows.
		
		Finally, a direct computation yields
		\[
		\begin{aligned}
			\frac{1}{|K|}\int_K \psi\, dxdy&=\int_0^1\psi H \,dx
			=\int_0^1\left(\int_0^1\psi(x)T_t(x)\,dx\right)d\beta(t)\\
			&=\int_0^1\frac{1+t(1-t)}{12}\,d\beta(t)\\
			&=\frac1{12}\beta([0,1])+\frac1{12}\int_0^1t(1-t)\,d\beta(t)\\
			&=\frac16+\frac1{12}\int_0^1t(1-t)\,d\beta(t).
		\end{aligned}
		\]
		This completes the proof.
	\end{proof}
	
	The next lemma shows, roughly speaking, that if the average of $\psi$ over $K$ exceeds $1/6$ by at most $O(|K|)$, then $K$ must develop nontrivial boundary layers near the endpoints of a diameter, producing a perimeter excess of order $|K|$.

	\begin{lemma}\label{lem:endpoint}
		Let $K$ be as in Lemma~\ref{lem:tent} with $|K|=a$. Set 
		\[
		I=\int_K\psi(x)\,dxdy,\qquad 
		B=\int_{\partial K}\psi(x)\,d\sigma,
		\qquad
		P=P(K).
		\]
		Then there is an absolute constant $C_B$ such that
		\begin{equation}\label{eq:B-estimate}
			0\le B-\frac13\le C_Ba^2.
		\end{equation}
		Furthermore, for  fixed $M>0$ there are constants $c_M>0$ and $a_M>0$ such that, if $0<a\le a_M$ and
		\begin{equation}\label{eq:near-worst}
			\frac Ia-\frac16\le Ma,
		\end{equation}
		then
		\begin{equation}\label{eq:P-excess}
			P-2\ge c_Ma.
		\end{equation}
	\end{lemma}
	
	\begin{proof}
		Observe that  $h:=p+q$ is concave and $h(0)=h(1)=0$. Comparison with the area yields 
		\begin{equation}\label{eq:pqmax}
			0\le p,q\le h\le2a.
		\end{equation}
		
		We now prove \eqref{eq:B-estimate}.  It follows from the definition of $B$ that 
		\begin{equation}\label{eq:B-minus}
			\begin{aligned}
				B-\frac13&=B-2\int_0^1\psi(x)\,dx\\
				&= \int_0^1\psi(x)\left(\sqrt{1+p'(x)^2}-1\right)dx+\int_0^1\psi(x)\left(\sqrt{1+q'(x)^2}-1\right)dx.
			\end{aligned}
		\end{equation}
		The right-hand side is nonnegative.
		We exploit the following elementary estimate: if $\phi\ge0$ is $C^2$ and concave on $[0,1]$, $\phi(0)=\phi(1)=0$, then
		\begin{equation}\label{eq:weighted-length-excess}
			\int_0^1x(1-x)\left(\sqrt{1+\phi'(x)^2}-1\right)dx
			\le \frac{1}{2} \|\phi\|_{L^\infty} ^2.
		\end{equation}
		Indeed, let $m$ be a point where $\phi$ attains its maximum. Then for $s\in [0,m]$, $\phi'(s)\ge0$ is nonincreasing.  Inequalities $\sqrt{1+\phi'(s)^2}-1\le \phi'(s)^2/2$ and $s(1-s)\le s$  yields 
		\[
		\begin{aligned}
			\int_0^m s(1-s)(\sqrt{1+\phi'(s)^2}-1)\,ds
			&\le \frac12\int_0^m s \phi'(s)^2\,ds\leq \frac{1}{2}\int_{0}^{m}\phi(s)\phi'(s)ds,
		\end{aligned}
		\]
		where we used the concavity of $\phi$ in the last inequality. The displayed integral is therefore at most $\|\phi\|_{L^\infty}^2/4$. The interval $[m,1]$ is handled in the same way after replacing $x$ by $1-x$ and contributes at most another $\|\phi\|_{L^\infty}^2/4$.  These two  observations imply \eqref{eq:weighted-length-excess}. Applying it to $p$ and $q$, we derive from \eqref{eq:pqmax}  that  \eqref{eq:B-estimate} holds.
		
		It remains to prove \eqref{eq:P-excess}. Lemma~\ref{lem:tent}  shows that  there is a nonnegative measure $\beta$ such that
		\[
		H(x)=\int_0^1T_t(x)\,d\beta(t),
		\qquad
		\beta([0,1])=2,
		\]
		and
		\[
		\frac Ia=\int_0^1\psi H
		=\frac16+\frac1{12}\int_0^1t(1-t)\,d\beta(t).
		\]
		The assumption \eqref{eq:near-worst}  entails 
		\begin{equation}\label{eq:beta-endpoint-small}
			\int_0^1t(1-t)\,d\beta(t)
			\le 12Ma.
		\end{equation}
		Choose $a_M=\tfrac{1}{800(M+1)}$ and $\ell=R_M a$ with $R_M=100(M+1)$.
		Then for $a\le a_M$,  one has  $$\ell<1/4.$$  Let
		\[
		\bar L=[0,\ell],
		\qquad
		\bar R=[1-\ell,1],
		\qquad
		\bar C=[\ell,1-\ell].
		\]
		On the interval $\bar C$, one has $$t(1-t)\ge \ell(1-\ell)\ge\ell/2.$$ Therefore \eqref{eq:beta-endpoint-small} yields
		\[
		\beta(\bar C)\le\frac{12Ma}{\ell/2}=\frac{24M}{R_M}\le\frac14.
		\]
		$\beta([0,1])=2$  means 
		\begin{equation}\label{eq:endpoint-beta-mass}
			\beta(\bar L)+\beta(\bar R)\ge\frac74.
		\end{equation}
		Moreover, we deduce that \[
		H(\ell)+H(1-\ell)
		\ge \frac34\bigl(\beta(\bar L)+\beta(\bar R)\bigr)
		\ge\frac{21}{16},
		\]
		because $T_t(\ell) \ge\frac34$ for $t\in \bar L$ and $T_t(1-\ell)\ge\frac34$ for $t\in \bar R$.

		We see that  at least one of the following is true:
		\begin{equation}\label{eq:height-endpoint}
			h(\ell)\ge\frac a2
			\qquad\text{or}\qquad
			h(1-\ell)\ge\frac a2.
		\end{equation}
		Suppose the first one holds. Then $p(\ell)+q(\ell)\ge a/2$. The upper boundary over $[0,\ell]$ has length at least $\sqrt{\ell^2+p(\ell)^2}$, and the lower boundary over $[0,\ell]$ has length at least $\sqrt{\ell^2+q(\ell)^2}$. Hence the contribution of this endpoint layer to $P-2$ is bounded below by
		\[
		\sqrt{\ell^2+p(\ell)^2}
		+\sqrt{\ell^2+q(\ell)^2}
		-2\ell.
		\]
		Since $z\mapsto\sqrt{\ell^2+z^2}$ is convex,  we find
		\[
		\begin{aligned}
			\sqrt{\ell^2+p(\ell)^2}
			+\sqrt{\ell^2+q(\ell)^2}
			-2\ell
			&\ge
			2\sqrt{\ell^2+\left(\frac a4\right)^2}-2\ell\\
			&=2a\left(\sqrt{R_M^2+\frac1{16}}-R_M\right).
		\end{aligned}
		\]
		Thus \eqref{eq:P-excess} holds with
		\[
		c_M=2\left(\sqrt{R_M^2+\frac1{16}}-R_M\right)>0.
		\]
		The second case $h(1-\ell)\ge a/2$ is identical  by using the right endpoint.
	\end{proof}
	
	We are now ready to prove Proposition~\ref{prop:gap}.

	\begin{proof}
		We adopt the notation of Lemma~\ref{lem:endpoint}.  By Lemma~\ref{lem:tent}, we set
		\[
		\frac Ia=\frac16+\eta,
		\qquad
		\eta\ge0.
		\]
		Lemma~\ref{lem:endpoint}   shows 
		\[
		B\le\frac13+C_Ba^2.
		\]
		Fix
		$
		M=2C_B+1.
		$
		Then we split into two cases.
		
		If $\eta\ge Ma$, then $P\ge2$ and hence
		\[
		\frac BP\le\frac{1/3+C_Ba^2}{2}
		=\frac16+\frac{C_B}{2}a^2.
		\]
		Moreover, for  sufficiently small  $a$,
		\[
		\frac Ia-\frac BP
		\ge Ma-\frac{C_B}{2}a^2
		\ge c a.
		\]
		
		If $\eta\le Ma$, then Lemma~\ref{lem:endpoint} also implies 
		\[
		P\ge2+c_Ma.
		\]
		Therefore, for $a$ sufficiently small,
		\[
		\frac BP
		\le\frac{1/3+C_Ba^2}{2+c_Ma}
		\le \frac16-\frac{c_M}{24}a.
		\]
		We further  obtain
		\[
		\frac Ia-\frac BP\ge\frac{c_M}{24}a.
		\]
		The proof is finished.
	\end{proof}
	
	With the establishment of Proposition \ref{prop:gap}, we now give the proof of Lemma \ref{lem:thin-lower}.
	
	\begin{proof}[Proof of Lemma \ref{lem:thin-lower}]
		Assume the statement is false. Then there exist convex domains $K_j$ such that
		\[
		\diam(K_j)=1,
		\qquad
		|K_j|=a_j\to0,
		\]
		and their Neumann torsion functions $v_j$ satisfy
		\begin{equation}\label{eq:osc-counter}
			\osc_{\partial K_j}v_j=o(a_j).
		\end{equation}
		After a rigid motion, choose diametral points $(0,0)$ and $(1,0)$. Then
		\[
		K_j=\{(x,y):0<x<1,\ -q_j(x)<y<p_j(x)\},
		\]
		where $p_j,q_j\ge0$ are concave and vanish at $0$ and $1$. Set $h_j=p_j+q_j$. Then
		\[
		\int_0^1h_j=a_j,
		\qquad
		\|h_j\|_{L^\infty} \le2a_j.
		\]
		Hence $K_j$ is contained in a strip of width $O(a_j)$.
		
		Let $\tau_j$ be the Dirichlet torsion function on $K_j$:
		\[
		\begin{cases}
			-\Delta\tau_j=1&\text{in }K_j,\\
			\tau_j=0&\text{on }\partial K_j.
		\end{cases}
		\]
		By comparison with the torsion function of a strip of width $C a_j$,
		\begin{equation}\label{eq:dirichlet-strip-barrier}
			0\le\tau_j\le C a_j^2.
		\end{equation}
		The difference $v_j-\tau_j$ is harmonic in $K_j$ and has boundary values equal to $v_j$. The maximum principle  provides 
		\[
		\osc_{K_j}(v_j-\tau_j)
		\le\osc_{\partial K_j}v_j=o(a_j).
		\]
		Together with \eqref{eq:dirichlet-strip-barrier}, this implies
		\begin{equation}\label{eq:global-osc-vj}
			\osc_{K_j}v_j=o(a_j).
		\end{equation}
		Since $v_j$ is defined up to an additive constant, we choose this constant so that
		\begin{equation}\label{eq:vj-Linfty-small}
			\|v_j\|_{L^\infty(K_j)}=o(a_j).
		\end{equation}
		
		Now take
		\[
		\psi(x)=x(1-x).
		\]
		The weak formulation   shows 
		\begin{equation}\label{eq:test-psi-weak}
			\int_{K_j}\nabla v_j\cdot\nabla\psi\,dxdy
			=
			\int_{K_j}\psi\,dxdy
			-\frac{a_j}{P_j}\int_{\partial K_j}\psi\,d\sigma.
		\end{equation}
		On the other hand, integrating by parts once, one has 
		\begin{equation}\label{eq:ibp-psi}
			\int_{K_j}\nabla v_j\cdot\nabla\psi\, dxdy
			=
			\int_{\partial K_j}v_j\psi'(x)\nu_x\,d\sigma
			-\int_{K_j}v_j\psi''(x)\,dxdy.
		\end{equation}
		We have $|\psi'|\le1$ and $|\psi''|=2$. Also, by convexity of $K_j$ and projection formula,
		\[
		\int_{\partial K_j}|\nu_x|\,d\sigma
		\le 2\|p_j\|_\infty+2\|q_j\|_\infty
		\le C a_j.
		\]
		Thus \eqref{eq:vj-Linfty-small} and \eqref{eq:ibp-psi} imply
		\[
		\left|\int_{K_j}\nabla v_j\cdot\nabla\psi\, dxdy\right|
		\le o(a_j)(Ca_j+2a_j)=o(a_j^2).
		\]
		It follows 	from \eqref{eq:test-psi-weak} that
		\begin{equation}\label{eq:moment-counter}
			\frac{1}{|K_j|}\int_{K_j}\psi\, dxdy - \frac{1}{P(K_j)}\int_{\partial K_j}\psi \, d\sigma=o(a_j).
		\end{equation}
		This contradicts Proposition~\ref{prop:gap}, and thus Lemma~\ref{lem:thin-lower} is proved.
	\end{proof}

	\section{Qualitative stability result: proof of Theorem \ref{thm:disk-conv}}
	\label{sec:qualitative-stability}
	We will combine the diameter compactness with a weak version of Serrin's overdetermined theorem to prove our qualitative stability result mentioned in the introduction.
	
	First, we   prove the following  diameter compactness result.
	
	\begin{theorem}\label{thm:diam-bound}
		Let $\Omega_k\subset\R^2$ be bounded convex domains with
		\[
		|\Omega_k|=\pi,
		\]
		and let $u_k$ be the normalized Neumann torsion function on $\Omega_k$, in the sense that $\int_{\partial \Omega_k}u_k\,d\sigma=0$. If
		\[
		O(\Omega_k)=\osc_{\partial\Omega_k}u_k\to0,
		\]
		then
		\[
		\sup_k\diam(\Omega_k)<\infty.
		\]
	\end{theorem}
	
	\begin{proof}
		Suppose, by contradiction, that along a subsequence
		\[
		L_k=\diam(\Omega_k)\to\infty.
		\]
		Define the rescaled domains
		\[
		K_k=\frac1{L_k}\Omega_k.
		\]
		Then
		\[
		\diam(K_k)=1,
		\qquad
		|K_k|=a_k=\frac{\pi}{L_k^2}\to0.
		\]
		If $v_k$ denotes the normalized Neumann torsion function on $K_k$, then the scaling relation is
		\[
		v_k(x)=\frac1{L_k^2}u_k(L_kx)
		\]
		up to an additive constant. Hence
		\[
		O(K_k)=\frac1{L_k^2}O(\Omega_k).
		\]
		Consequently
		\[
		\frac{O(K_k)}{|K_k|}
		=
		\frac{L_k^{-2}O(\Omega_k)}{\pi L_k^{-2}}
		=
		\frac{O(\Omega_k)}{\pi}\to0.
		\]
		This contradicts Lemma~\ref{lem:thin-lower}. Therefore the diameters are uniformly bounded.
	\end{proof}
	We now proceed to prove Theorem~\ref{thm:disk-conv}.
	\begin{proof}[Proof of Theorem~\ref{thm:disk-conv}]
		By Theorem~\ref{thm:diam-bound}, the diameters of \(\Omega_k\) are uniformly bounded. After suitable translations and passage to a subsequence, the Blaschke selection theorem~\cite[Theorem 1.8.7]{Schneider2014} implies that there is a convex domain $\Omega_\infty$ such that 
		\[
		\Omega_k \longrightarrow \Omega_\infty
		\]
		in the Hausdorff topology, with all the domains contained in a fixed ball. Since the areas are fixed and positive, \(\Omega_\infty\) has nonempty interior, and convex-body convergence guarantees
		\begin{equation}\label{eq:area_per}
			\lvert\Omega_k\rvert\longrightarrow\lvert\Omega_\infty\rvert,
			\qquad
			P(\Omega_k)\longrightarrow P(\Omega_\infty).
		\end{equation}
		The fixed area and the uniform diameter bound imply a uniform positive lower bound for the inradii. Hence the convex domains \(\Omega_k\) have a uniform Lipschitz character and admit uniformly bounded Sobolev extension operators.
		
		Let \(u_k\) be the Neumann torsion function, and choose a constant \(m_k\) such that, with
		\[
		w_k:=u_k-m_k,
		\]
		one has
		\[
		0\leq w_k\leq O(\Omega_k)
		\qquad\text{on }\partial\Omega_k.
		\]
		Fix \(p>2\). By the \(W^{1,p}\) estimate introduced by Geng--Shen~\cite[Theorem~1.2]{GengShen}, the uniform Lipschitz geometry and the uniformly bounded data   provide 
		\[
		\lVert\nabla w_k\rVert_{L^p(\Omega_k)}\leq C_p.
		\]
		The boundary normalization, together with a uniform Poincar\'e--trace inequality, controls the additive constants. Hence
		\[
		\lVert w_k\rVert_{W^{1,p}(\Omega_k)}\leq C_p.
		\]
		Let \(\widetilde w_k\) be uniformly bounded extensions to a fixed ball \(B_R\) containing all \(\Omega_k\). After passing to a subsequence,
		\[
		\widetilde w_k\rightharpoonup \widetilde w_\infty
		\qquad\text{weakly in }W^{1,p}(B_R),
		\]
		and, by the compact Morrey embedding,
		\[
		\widetilde w_k\longrightarrow \widetilde w_\infty
		\qquad\text{uniformly on }\overline{B_R}.
		\]
		
		Set \(w_\infty:=\widetilde w_\infty\vert_{\Omega_\infty}\).
		For  fixed \(\varphi\in C^1(\mathbb R^2)\), the characteristic functions \(\chi_{\Omega_k}\) converge strongly to \(\chi_{\Omega_\infty}\) in  \(L^q\) space $(q<\infty)$, while convex-body convergence shows weak convergence of the boundary measures. Passing to the limit in
		\[
		\int_{\Omega_k}\nabla w_k\cdot\nabla\varphi\,dx
		=
		\int_{\Omega_k}\varphi\,dx
		-
		\frac{\lvert\Omega_k\rvert}{P(\Omega_k)}
		\int_{\partial\Omega_k}\varphi\,d\mathcal H^1
		\]
		yields
		\begin{equation}\label{eq:weak-serrin-limit}
			\int_{\Omega_\infty}\nabla w_\infty\cdot\nabla\varphi\,dx
			=
			\int_{\Omega_\infty}\varphi\,dx
			-
			\frac{\lvert\Omega_\infty\rvert}{P(\Omega_\infty)}
			\int_{\partial^*\Omega_\infty}\varphi\,d\mathcal H^1.
		\end{equation}
		Similar limit argument is also used in the first author's PhD thesis \cite[Chapter 4-5]{Li18}.
		
		It remains to identify the boundary trace. Given \(x\in\partial\Omega_\infty\), choose \(x_k\in\partial\Omega_k\) with \(x_k\to x\). The uniform convergence and the boundary normalization imply
		\[
		\widetilde w_\infty(x)
		=
		\lim_{k\to\infty}w_k(x_k)
		=0.
		\]
		Thus \(w_\infty=0\) on \(\partial\Omega_\infty\), and in particular
		\[
		w_\infty\in H^1_0(\Omega_\infty).
		\]
		Equation~\eqref{eq:weak-serrin-limit} is therefore the weak formulation of Serrin's overdetermined torsion problem with constant Dirichlet datum and constant normal derivative.
		
		By the rough-domain theorem of Figalli--Zhang~\cite{FigalliZhang2025}, \(\Omega_\infty\) is a ball. Since \(\lvert\Omega_\infty\rvert=\pi\), it is a unit disk. Thus every Hausdorff subsequential limit, after translation, is a unit disk, and consequently
		\[
		\inf_{z\in\mathbb R^2}d_H\bigl(\Omega_k,B_1(z)\bigr)\longrightarrow0.
		\]
	\end{proof}
	
	\begin{remark}[Rough rigidity as the compactness endpoint]
		The use of Figalli--Zhang is essential to the preceding compactness argument. Indeed, the limiting convex domain need not be smooth, and the passage to the limit yields only the weak overdetermined formulation \eqref{eq:weak-serrin-limit}. The classical smooth Serrin theorem therefore cannot be invoked directly, whereas the rough-domain theorem of Figalli--Zhang applies precisely at this level of regularity.
		
		More conceptually, the proof illustrates a useful route from rigidity to stability: once the oscillation  produces geometric precompactness and the associated PDE passes to the limit, a rigidity theorem formulated at the natural weak regularity of the limit identifies  every limit with vanishing boundary oscillation. In the present problem, the rough-domain Serrin theorem is exactly the result that closes this mechanism.
	\end{remark}

	\section{A quantitative consequence in the convex class}
	\label{sec:quantitative-full-convex}
	In this section, we establish a quantitative result in the convex class and produce the preliminary
	rate. Throughout this section, we assume that $\Omega\subset \mathbb{R}^{2}$ is a bounded convex domain with $|\Omega|=\pi$, and we write
	\[
	r_\Omega:=\sup\{r>0:\text{there exists }x\in\Omega\text{ with }B_r(x)\subset\Omega\}
	\]
	for its inradius and
	\[
	R_\Omega:=\inf\{R>0:\text{there exists }x\in\R^2\text{ with }\Omega\subset B_R(x)\}
	\]
	for its circumradius. The following compactness consequence will be used below.
	
	\begin{proposition}
		\label{prop:uniform-geometry}
		There exist $\delta_*>0$, $D_*>0$, and $r_*>0$ such that every convex domain
		$\Om\subset\R^2$ with $|\Om|=\pi$ and $O(\Om)\le\delta_*$ satisfies
		\begin{equation}\label{eq:uniform-geometry}
			\diam(\Om)\le D_*\qquad\text{and}\qquad r_\Om\ge r_*.
		\end{equation}
	\end{proposition}
	
	\begin{proof}
		The diameter upper bound follows from Theorem~\ref{thm:diam-bound}. Indeed, if no such universal bound $D_*$ existed,  there would exist a sequence of convex domains $\Omega_j$ such that
		\begin{equation*}
			O(\Omega_j)\to0  \qquad \text { as } \qquad\diam(\Omega_j)\to\infty
		\end{equation*}
		contradicting Theorem~\ref{thm:diam-bound}.   Some elementary calculations show  
		\begin{equation}\label{eq:r-upbad}
			\begin{aligned}
				\pi&=|\Om|\le r_\Om P(\Om)\le \pi  r_\Om \diam(\Om)\le \pi D_*  r_\Om .
			\end{aligned}
		\end{equation}
		Then the universal positive lower bound $r_*$ follows.
	\end{proof}

	For sufficiently small $O(\Omega)$, we have the following tame quantitative result. 
	\begin{theorem}\label{thm:main}
		There exist constants $C>1$ and $0<\delta_0\le\delta_*$, depending only on
		the constants in \eqref{eq:uniform-geometry}, such that every bounded convex domain $\Om\subset\R^2$ with $|\Om|=\pi$ and
		$\delta=O(\Om)\le\delta_0$ satisfies
		\begin{equation}\label{eq:log-estimate}
			R_\Om-r_\Om
			\le C\,\delta^{1/2}\log\frac{C}{\delta}.
		\end{equation}
	\end{theorem}

	The quantitative argument uses the reverse-Serrin identity of Magnanini--Molinarolo--Poggesi~\cite{MagnaniniMolinaroloPoggesi2024}, together with interpolation tools developed by Magnanini--Poggesi~\cite{MP19,MP201,MP202,MagnaniniPoggesiInterpolation}. The key additional ingredients are a new tangential gradient estimate for the Neumann torsion function, established in Lemma~\ref{lem:boundary-gradient} below, and the compactness mechanism developed in the previous section. Together, these ingredients eliminate the need for the uniform interior and exterior sphere conditions imposed in earlier works.
	
	To match the notation of the reverse-Serrin identity in Lemma \ref{lem:MMP-identity},
	we reverse the sign and normalize the source term. Let $u$ be the Neumann torsion function in $\Omega$ and set
	\begin{equation}\label{eq:v-def}
		v:=-2u+c_0,
	\end{equation}
	where $c_0$ is chosen so that
	\[
	\min_{\partial\Om}v=0.
	\]
	Then  one has 
	\begin{equation}\label{eq:v-problem}
		\Delta v=2\quad\text{in }\Om,
		\qquad
		\pnu v=c_\Om:=\frac{2\pi}{P(\Om)}\quad\text{on }\partial\Om.
	\end{equation}

	The following lemma is the first essential point. It converts small oscillation of
	the Dirichlet trace into a pointwise estimate for its tangential derivative.
	
	\begin{lemma}\label{lem:boundary-gradient}
		Let $\Om$ be smooth and strictly convex, with $O(\Om)=\varepsilon/2$, where $0<\varepsilon<\delta_*$ and $\delta_*$ is as in Proposition~\ref{prop:uniform-geometry}.Then the solution of
		\eqref{eq:v-problem}  satisfies
		\begin{equation}\label{eq:tangential-bound}
			|\nabla_{\partial\Om}v|^2\le 2v\le2\varepsilon
			\qquad\text{on }\partial\Om.
		\end{equation}
		Moreover,
		\begin{equation}\label{eq:global-gradient}
			\|\nabla v\|_{L^\infty(\Om)}^2\le c_\Om^2+2\varepsilon.
		\end{equation}
	\end{lemma}
	
	\begin{proof}
		From the definition of \eqref{eq:v-def},  it follows that 
		\begin{equation}\label{eq:eps-def}
			0\le v\le\varepsilon:=\osc_{\partial\Om}v 
			\quad\text{on }\partial\Om.
		\end{equation}
		Note that	 $v$ is subharmonic, we adopt  the maximum principle to get  
		\begin{equation}\label{eq:v-upper}
			v\le\varepsilon\quad\text{in }\Om.
		\end{equation}
		
		To estimate the tangential derivative, 	define  the following    P-function
		\[
		\mathcal P:=\frac12|\nabla v|^2-v.
		\]
		Since $\Delta v=2$, the two-dimensional Newton inequality gives
		\[
		\Delta\mathcal P=|D^2v|^2-2\ge\frac{(\Delta v)^2}{2}-2=0.
		\]
		That is, $\mathcal P$ is also subharmonic.		We claim that
		\begin{equation}\label{eq:P-upper}
			\mathcal P\le\frac{c_\Om^2}{2}\quad\text{in }\overline\Om.
		\end{equation}
		Indeed,	assume that  $\mathcal P$ is constant, and choose  $x_* \in \partial\Om$   satisfying  $$v(x_*)=\min_{\partial\Om}v=0.$$
		For the   smooth and strictly convex  domain $\Om$, 
		let $\tau$ be a unit tangent, $\nu$ be the exterior unit normal, and  $\kappa>0$ denote the curvature.  Then we use the
		convention
		\[
		D_\tau\nu=\kappa\tau,
		\qquad
		D_\tau\tau=-\kappa\nu.
		\]
		Differentiating the constant Neumann condition along the boundary  means 
		\begin{equation}\label{eq:mixed-hessian}
			v_{\tau\nu}=-\kappa v_\tau.
		\end{equation}
		Hence,  we see from \eqref{eq:v-problem} that 
		\[
		|\nabla v(x_* )|^2 = v_{\tau}(x_* )^2 + v_{\nu}(x_* )^2 = c_{\Omega}^2.
		\]
		We further have
		\[
		\mathcal{P}(x_* )= \frac{c_{\Omega}^2}{2}\qquad\text{and }\qquad\mathcal P=c_\Om^2/2.
		\]
		
		On the other hand, 	 
		if  $\mathcal P$ is not constant, then  it attains its maximum at some point  $x_0\in\partial\Om$. 
		At $x_0$,  we infer from \eqref{eq:v-problem} and \eqref{eq:mixed-hessian} that 
		\begin{equation*}
			\begin{aligned}
				0&=\mathcal P_\tau=v_\tau v_{\tau\tau}+c_\Om v_{\tau\nu}-v_\tau
				=v_\tau\bigl(v_{\tau\tau}-1-c_\Om\kappa\bigr).
			\end{aligned}
		\end{equation*}
		If $v_\tau(x_0)\ne0$, then
		$v_{\tau\tau}=1+c_\Om\kappa$.  Using \eqref{eq:v-problem} and \eqref{eq:mixed-hessian} again, one has 
		\[
		\mathcal P_\nu
		=c_\Om v_{\nu\nu}+v_\tau v_{\tau\nu}-c_\Om
		=-\kappa\bigl(c_\Om^2+v_\tau^2\bigr)<0.
		\]
		This contradicts the Hopf  lemma, which  implies that 
		$\mathcal P_\nu(x_0)>0$ at a boundary maximum of a nonconstant subharmonic function. Thus $v_\tau(x_0)=0$.   We infer  from   \eqref{eq:eps-def} that 
		\[
		\max_{\overline\Om}\mathcal P
		=\mathcal P(x_0)=\frac{c_\Om^2}{2}-v(x_0)
		\le\frac{c_\Om^2}{2}.
		\]
		By virtue of 
		\[
		\mathcal P=\frac12\bigl(c_\Om^2+|\nabla_{\partial\Om}v|^2\bigr)-v, \quad\text{on }\partial\Om,
		\]
		\eqref{eq:P-upper}  shows that 
		$|\nabla_{\partial\Om}v|^2\le2v$, and so
		\eqref{eq:tangential-bound} holds.  Finally, \eqref{eq:v-upper} and \eqref{eq:P-upper}  indicate 
		\[
		|\nabla v|^2\le c_\Om^2+2v\le c_\Om^2+2\varepsilon,
		\]
		which is \eqref{eq:global-gradient}.
	\end{proof}
	
	The next lemma derives the uniform depth estimate of minimal point of $v$.

	\begin{lemma}\label{lem:minimum-depth}
		Let $\Omega$ be as in Lemma~\ref{lem:boundary-gradient}.
		There exist $d_0>0$ and $\varepsilon_0>0$, depending only on $r_*$ and $D_*$,
		such that, if $\varepsilon\le\varepsilon_0$ and $z\in\Om$ is a minimum point
		of $v$, then
		\begin{equation}\label{eq:z-depth}
			\dist(z,\partial\Om)\ge d_0.
		\end{equation}
	\end{lemma}
	
	\begin{proof}
		Let $x_*$ be the center of an inball of radius at least $r_*$. 
		Let $\tau_\Om$ be the solution of
		\begin{equation}\label{eq:tau-problem}
			\Delta\tau_\Om=2\quad\text{in }\Om,
			\qquad
			\tau_\Om=0\quad\text{on }\partial\Om.
		\end{equation}
		Thus $-\tau_\Om/2$ is the usual positive Dirichlet torsion function. By domain monotonicity for the positive Dirichlet torsion function, we derive 
		\[
		\tau_\Om(x_*)\le-\frac{r_*^2}{2}.
		\]
		The function $v-\tau_\Om$ is harmonic and its boundary values lie in
		$[0,\varepsilon]$. Hence
		\[
		0\le v-\tau_\Om\le\varepsilon\quad\text{in }\Om.
		\]
		If $\varepsilon\le r_*^2/4$, then
		\begin{equation}\label{eq:min-negative}
			v(z)\le v(x_*)\le-\frac{r_*^2}{4}.
		\end{equation}
		By the isoperimetric inequality, one has 
		\begin{equation}\label{eq:c-om<1}
			c_\Om=2\pi/P(\Om)\le1.
		\end{equation}   
		So	it follows from Lemma~\ref{lem:boundary-gradient} and \eqref{eq:c-om<1}  that a uniform bound
		\begin{equation*}
			\|\nabla v\|_\infty\le L_*:=\sqrt{1+2\varepsilon}.
		\end{equation*}
		Applying $v\ge0$ on $\partial\Om$  and \eqref{eq:min-negative},  we find 
		\[
		\dist(z,\partial\Om)
		\ge\frac{r_*^2}{4L_*}=:d_0.
		\]
	\end{proof}
	\begin{remark}
		The hot-spot localization result of Magnanini--Poggesi
		\cite{MagnaniniPoggesiHotSpots} concerns maxima of the positive
		Dirichlet torsion function. Lemma~\ref{lem:minimum-depth} gives an
		analogous interior-depth estimate for a minimum point of the
		constant-flux Neumann solution, using $\tau_\Omega$ as a comparison
		function and the gradient bound from Lemma~\ref{lem:boundary-gradient}.
	\end{remark}

	Fix a minimum point $z$ as in Lemma~\ref{lem:minimum-depth}, and define
	\begin{equation}\label{eq:q-h-def}
		q(x):=\frac12|x-z|^2,
		\qquad
		h:=q-v.
	\end{equation}
	Since $\Delta q=\Delta v=2$, the function $h$ is harmonic in $\Om$. Moreover,
	\begin{equation}\label{eq:grad-h-z}
		\nabla h(z)=0.
	\end{equation}
	Let $\varepsilon$ be defined in \eqref{eq:eps-def} and 
	\[
	\overline v:=\max_{\partial\Om}v=\varepsilon,
	\]
	then	we use the following integral identity for the reverse Serrin problem. It is
	the planar case of the identity of Magnanini--Molinarolo--Poggesi.
	
	\begin{lemma}[Reverse Serrin identity]\label{lem:MMP-identity}
		Let $\Om$ be a bounded $C^2$ domain in $\R^2$, let $v$ satisfy
		$\Delta v=2$ in $\Om$ and $v_\nu=c_\Om$ on $\partial\Om$, and let $h=q-v$ be
		as in \eqref{eq:q-h-def}. Then
		\begin{align}
			\int_\Om(\overline v-v)|D^2h|^2\dd x
			={}&\frac12\int_{\partial\Om}|\nabla_{\partial\Om}v|^2h_\nu\dd s
			\label{eq:MMP-identity}\\
			&+\int_{\partial\Om}(\overline v-v)(c_\Om\kappa-2)h_\nu\dd s\notag\\
			&-\int_{\partial\Om}(\overline v-v)\kappa
			|\nabla_{\partial\Om}v|^2\dd s.\notag
		\end{align}
		Here $\kappa$ is the curvature with respect to the exterior normal.
	\end{lemma}
	
	\begin{proof}
		The two-dimensional specialization of the identity in Magnanini--Molinarolo--Poggesi~\cite[Corollary~3.3]{MagnaniniMolinaroloPoggesi2024} applies to $q$ and $v$, which have the same constant Laplacian, with harmonic difference $h=q−v$.
		In dimension two the mean curvature is the scalar curvature
		$\kappa$, and $\langle(\nabla\nu)\nabla_\Gamma v,\nabla_\Gamma v\rangle
		=\kappa|\nabla_\Gamma v|^2$, giving exactly
		\eqref{eq:MMP-identity}.
	\end{proof}
	
	The next lemma presents the weighted hessian bound of $h$, as a consequence of Lemma \ref{lem:boundary-gradient} and Lemma \ref{lem:MMP-identity}.
	
	\begin{lemma}\label{lem:weighted-hessian}
		Let $\Omega$ be as in Lemma~\ref{lem:boundary-gradient}. There exists $C>0$, depending only on $D_*$, such that
		\begin{equation}\label{eq:weighted-Hessian-estimate}
			\int_\Om d(x)^2|D^2h(x)|^2\dd x\le C\varepsilon.
		\end{equation}
	\end{lemma}
	
	\begin{proof}
		First, the elementary interior-ball comparison yields
		\begin{equation}\label{eq:torsion-distance-lower}
			\overline v-v(x)\ge\frac12\dist(x,\partial\Om)^2
			\qquad(x\in\Om).
		\end{equation}
		Indeed, $w=\overline v-v$ satisfies $-\Delta w=2$ and $w\ge0$ on
		$\partial\Om$. For fixed $x\in\Om$, compare $w$  with
		\[
		\psi_x(y)=\frac{r^2-|y-x|^2}{2},
		\]
		in $B_{r}(x)$, where $r=\dist(x,\partial\Om)$.
		Since $\psi_x$ solves $-\Delta\psi=2$ and vanishes on  $\partial B_{r}(x)$,  \eqref{eq:torsion-distance-lower} holds.

		It remains to estimate the right-hand side of
		\eqref{eq:MMP-identity}.  $|x-z|\le D_*$ and $c_\Om\le1$  entails 
		\begin{equation}\label{eq:h-nu-bound}
			|h_\nu|=|(x-z)\cdot\nu-c_\Om|\le D_*+1
			\quad\text{on }\partial\Om.
		\end{equation}
		For a planar convex domain, one has 
		\[
		\int_{\partial\Om}\kappa\dd s=2\pi,
		\qquad
		P(\Om)\le\pi D_*.
		\]
		Moreover, \eqref{eq:eps-def} yields that 
		\begin{equation*}
			0\le\overline v-v\le\varepsilon \quad\text{on }\partial\Om.
		\end{equation*}
		Then using	\eqref{eq:tangential-bound}  and \eqref{eq:h-nu-bound}, we obtain
		\begin{align*}
			\left|\frac12\int_{\partial\Om}
			|\nabla_{\partial\Om}v|^2h_\nu\dd s\right|
			&\le C\varepsilon,\\
			\left|\int_{\partial\Om}(\overline v-v)
			(c_\Om\kappa-2)h_\nu\dd s\right|
			&\le C\varepsilon
			\left(c_\Om\int_{\partial\Om}\kappa\dd s+2P(\Om)\right)
			\le C\varepsilon,\\
			\left|\int_{\partial\Om}(\overline v-v)\kappa
			|\nabla_{\partial\Om}v|^2\dd s\right|
			&\le2\varepsilon^2\int_{\partial\Om}\kappa\dd s
			\le C\varepsilon^2.
		\end{align*}
		Hence \eqref{eq:MMP-identity} ensures 
		\[
		\int_\Om(\overline v-v)|D^2h|^2\dd x\le C\varepsilon.
		\]
		Combining this with \eqref{eq:torsion-distance-lower} proves
		\eqref{eq:weighted-Hessian-estimate}.
	\end{proof}

	The geometry in \eqref{eq:uniform-geometry} gives a uniform analytic class.
	Indeed, a convex domain containing a ball of radius $r_*$ and having diameter
	at most $D_*$ satisfies a uniform John condition and a uniform interior cone
	condition, with parameters depending only on $r_*$ and $D_*$. We shall use two
	standard inequalities on this class to derive the oscillation of $h$ from the weighted hessian estimate \eqref{eq:weighted-Hessian-estimate}. 
	
	The first inequality is as below, which follows from Dur\'{a}n~\cite[Theorem~5.3]{Duran12}.
	\begin{lemma}\label{lem:improved-Poincare}
		Let $\Omega\subset\R^2$ be a bounded domain of diameter $R$ that is star-shaped with respect to a ball $B\subset\Omega$ of radius $\rho$. Then every $f\in H^1(\Omega)$ with zero mean satisfies
		\begin{equation}\label{eq:improved-Poincare}
			\|f\|_{L^2(\Omega)}
			\le C\|\dist(x,\partial\Om) \nabla f\|_{L^2(\Omega)}
		\end{equation}
		where $C>0$ depends only on $R$ and $\rho$.
	\end{lemma}
	The oscillation estimate of Magnanini--Poggesi~\cite[Theorem~2.7]{MagnaniniPoggesiInterpolation} implies the second. 
	\begin{lemma}\label{lem:critical-interpolation}
		Let $\Omega\subset\R^2$ satisfy the $(\theta,a)$-uniform interior cone condition. Then every
		$f\in W^{1,\infty}(\Omega)$ satisfies
		\begin{equation}\label{eq:critical-interpolation}
			\osc_{\overline \Omega}f
			\le C\|\nabla f\|_{L^2(\Omega)}
			\log\left(M\frac{\|\nabla f\|_{L^\infty(\Omega)}}
			{\|\nabla f\|_{L^2(\Omega)}}\right),
		\end{equation}
		where the positive constants $C$ and $M$ depend only on $\theta$, $a$, and $|\Omega|$.
	\end{lemma}
	
	We now derive the $L^2$ and $L^\infty$ estimates of $\nabla h$.
	
	\begin{lemma}
		\label{lem:small-gradient-h}
		Under the assumptions of Lemma~\ref{lem:minimum-depth}, there exists a constant $C>0$, depending only on $r_*$ and $D_*$, such that, for every $\varepsilon<\varepsilon_0$,
		\begin{equation}\label{eq:L2-gradient-h}
			\|\nabla h\|_{L^2(\Om)}\le C\varepsilon^{1/2}.
		\end{equation}
		Moreover,
		\begin{equation}\label{eq:Linfty-gradient-h}
			\|\nabla h\|_{L^\infty(\Om)}\le C.
		\end{equation}
	\end{lemma}
	
	\begin{proof}
		For each $i$, let $a_i$ denote the mean value of $\partial_i h$ over $\Omega$; thus $\partial_i h-a_i$ has zero mean. From  Lemma~\ref{lem:improved-Poincare}  we deduce 
		\[
		\|\partial_i h-a_i\|_{L^2(\Om)}
		\le C\|d\nabla(\partial_i h)\|_{L^2(\Om)}.
		\]
		Let $a=(a_1,a_2)$, then  
		Lemma~\ref{lem:weighted-hessian}  presents
		\begin{equation}\label{eq:grad-minus-a}
			\|\nabla h-a\|_{L^2(\Om)}\le C\varepsilon^{1/2}.
		\end{equation}
		
		Owing to Lemma~\ref{lem:minimum-depth}, the ball
		$B:=B_{d_0/2}(z)$ is contained in $\Om$. Each component $\partial_i h$ is
		harmonic, and  then \eqref{eq:grad-h-z} entails  $\partial_i h(z)=0$. The mean-value
		property therefore yields
		\[
		0=\fint_B\partial_i h,
		\qquad
		|a_i|=\left|\fint_B(a_i-\partial_i h)\right|
		\le |B|^{-1/2}\|\partial_i h-a_i\|_{L^2(B)}.
		\]
		Thus $|a|\le C\varepsilon^{1/2}$.Furthermore, \eqref{eq:L2-gradient-h} follows from 
		\eqref{eq:grad-minus-a}.
		By the definition of  \eqref{eq:q-h-def}, one has
		\[
		\nabla h=x-z-\nabla v.
		\]
		So we derive from \eqref{eq:uniform-geometry} and \eqref{eq:global-gradient} that 
		\eqref{eq:Linfty-gradient-h}.
	\end{proof}
	
	\begin{corollary} 
		\label{cor:oscillation-h}
		Under the assumptions of Lemma~\ref{lem:minimum-depth}, there exists $C>1$, depending only on the constants in \eqref{eq:uniform-geometry}, such that
		\begin{equation}\label{eq:osc-h}
			\osc_{\overline\Om}h
			\le C\varepsilon^{1/2}\log\frac{C}{\varepsilon}
		\end{equation}
		for all $\varepsilon<\varepsilon_{0}$.
	\end{corollary}
	
	\begin{proof}
		Applying Lemma~\ref{lem:critical-interpolation} to $h$  and appealing to \eqref{eq:L2-gradient-h}--\eqref{eq:Linfty-gradient-h},  we obtain   \eqref{eq:osc-h} after possibly enlarging the constant $C$.
	\end{proof}
	
	Now we are ready to prove the quantitative estimate in the section.
	
	\begin{proof}[Proof of Theorem~\ref{thm:main}]
		We first assume that $\Omega$ is smooth and strictly convex, so that the assumptions of Lemma~\ref{lem:minimum-depth} hold. Let $z$ be the minimum point of $v$ chosen above, and set
		\[
		\rho_i:=\min_{x\in\partial\Om}|x-z|,
		\qquad
		\rho_e:=\max_{x\in\partial\Om}|x-z|.
		\]
		\eqref{eq:eps-def}  and  \eqref{eq:q-h-def}   lead to
		\begin{align}
			\frac12(\rho_e^2-\rho_i^2)
			&=\osc_{\partial\Om}q\le\osc_{\partial\Om}h+\osc_{\partial\Om}v\notag\\
			&\le\osc_{\overline\Om}h+\varepsilon.
			\label{eq:radii-q}
		\end{align}
		By the definition of $\rho_e$, we  see that $\Om\subset B_{\rho_e}(z)$. Therefore
		\begin{equation*}
			\pi=|\Om|\le  |B_{\rho_e}(z)|=\pi  \rho_e^2.
		\end{equation*}
		i.e., $\rho_e\ge1$.  We further see that  $\rho_e+\rho_i\ge1$  and
		\[
		\rho_e-\rho_i
		\le\rho_e^2-\rho_i^2
		\le2\osc_{\overline\Om}h+2\varepsilon.
		\]
		
		The definitions of the global inradius and circumradius imply
		\[
		R_\Om\le\rho_e,
		\qquad
		r_\Om\ge\rho_i,
		\]
		and so
		\[
		R_\Om-r_\Om\le\rho_e-\rho_i.
		\]
		Corollary~\ref{cor:oscillation-h} now yields
		\eqref{eq:log-estimate} 	in view of $\varepsilon=2 O(\Om)$.

		For completeness, we indicate the standard passage to arbitrary convex
		domains. Let $\Om_j$ be smooth strictly convex bodies converging to $\Om$ in
		the Hausdorff topology, followed by a homothety so that $|\Om_j|=\pi$.
		Support-function mollification provides such an approximation. Then
		\[
		r_{\Om_j}\to r_\Om,
		\qquad
		R_{\Om_j}\to R_\Om,
		\qquad
		P(\Om_j)\to P(\Om).
		\]
		The normalized Neumann solutions converge under this uniform convex-domain
		approximation, and the boundary oscillations satisfy
		\[
		O(\Om_j)\to O(\Om).
		\]
		Indeed, let $u_j$ be the normalized Neumann torsion function on $\Omega_j$, with $\min_{\partial\Omega_j}u_j=0$. By Geng--Shen~\cite{GengShen}, for any fixed $p>2$ the functions $u_j$ are uniformly bounded in $W^{1,p}(\Omega_j)$. By the uniform extension property of the domains $\Omega_j$, suitable extensions $\widetilde u_j$ are uniformly bounded in $W^{1,p}_0(\mathbb{R}^2)$ and therefore have a uniform H\"older modulus of continuity. After passing to a subsequence, $\widetilde u_j$ converges locally uniformly to a function $u_\infty$, whose restriction to $\Omega$ is precisely the normalized Neumann torsion function $u$. Uniform convergence, together with Hausdorff convergence of the boundaries, shows 
		\[
		\osc_{\partial\Omega}u
		=\lim_{j\to\infty}\osc_{\partial\Omega_j}u_j.
		\]
		Hence $O(\Omega_j)\to O(\Omega)$ as $j\to\infty$.

		Applying \eqref{eq:log-estimate} to $\Om_j$ and passing to the limit
		gives the estimate for $\Om$, completing the proof of Theorem~\ref{thm:main} for arbitrary bounded planar convex domains.
		
	\end{proof}
	
	\begin{remark}
		The argument of this section deliberately uses only the quadratic lower bound
		$\varepsilon-v\ge d^2/2$ and a critical $L^2$ interpolation estimate. It therefore
		produces the preliminary rate in Theorem~\ref{thm:main}. In the next section we
		use two additional pieces of structure---a uniform linear boundary-growth bound
		and the fact that $h_\nu$ is itself controlled by the radial defect---to upgrade
		this estimate to the sharp linear scale.
	\end{remark}

	\section{Sharp stability in the planar convex class: proof of Theorem \ref{thm:intro-quantitative}}
	\label{sec:sharp-stability}
	
	We now upgrade the qualitative compactness and the estimates of the preceding section to the optimal linear scale stated in Theorem \ref{thm:intro-quantitative}.

	First, we need the following estimate, which is the uniform John-domain version of the two-dimensional estimate underlying Lemma~5.2 of Magnanini--Molinarolo--Poggesi~\cite{MagnaniniMolinaroloPoggesi2024}. We state it separately because the constants in our full convex class must not depend on the curvature of a smooth approximation.
	
	\begin{lemma}\label{lem:uniform-weighted-oscillation}
		Let $\Omega\subset\R^2$ be a bounded convex domain with
		\[
		\diam(\Omega)\le D,
		\]
		and let $z\in\Omega$ satisfy
		\[
		\dist(z,\partial\Omega)\ge d_0>0.
		\]
		If $h$ is harmonic in $\Omega$ and $\nabla h(z)=0$, then
		\begin{equation}\label{eq:uniform-weighted-oscillation}
			\osc_{\overline\Omega}h
			\le C
			\left(\int_\Omega d(x,\partial\Omega)|D^2h(x)|^2\,dx\right)^{1/2},
		\end{equation}
		where $C$ depends only on $D$ and $d_0$.
	\end{lemma}
	
	\begin{proof}
		For $x\in\Omega$, consider the segment
		\[
		\gamma_x(t)=(1-t)x+tz,
		\qquad 0\le t\le1.
		\]
		Then one has $$|\gamma_x(t)-x|=t|z-x|\le tD.$$ Moreover,  
		The distance to the boundary is concave  and satisfies  
		\[
		d(\gamma_x(t),\partial\Omega)
		\ge (1-t)d(x,\partial\Omega)+t d(z,\partial\Omega)
		\ge t d_0.
		\]
		Hence, the domain is an $L_0$-John domain with base point $z$, with $L_0$ controlled only by $D/d_0$.
		
		We apply the weighted Hardy--Poincar\'e estimate for harmonic functions in the form of \cite[Corollary~2.3 and Remark~2.4]{MP201}, with
		\[
		N=2,\qquad r=4,\qquad p=2,\qquad \alpha=\frac12,
		\]
		which  joint with	the condition $\nabla h(z)=0$   ensures 
		\begin{equation}\label{eq:L4-from-weighted-Hessian}
			\|\nabla h\|_{L^4(\Omega)}
			\le C\|d(\cdot,\partial\Omega)^{1/2}D^2h\|_{L^2(\Omega)}.
		\end{equation}
		Here the constant $C$ depends only on $D$ and $d_0$,  and the volume dependence is absorbed  due to  $B_{d_0}(z)\subset\Omega\subset B_D(z)$.
		
		Moreover, the convex hull of any boundary point and the ball $B_{d_0/2}(z)$ is contained in $\Omega$. Thus these domains satisfy a uniform interior cone condition, again with parameters controlled only by $D$ and $d_0$. The corresponding uniform Morrey estimate yields
		\[
		\osc_{\overline\Omega}h
		\le C\|\nabla h\|_{L^4(\Omega)}.
		\]
		Combining this inequality with \eqref{eq:L4-from-weighted-Hessian} proves \eqref{eq:uniform-weighted-oscillation}.
	\end{proof}
	
	The next lemma is the crucial new ingredient for obtaining sharp stability estimate beyond the method in Section \ref{sec:quantitative-full-convex}. It replaces the previous elementary quadratic lower bound \eqref{eq:torsion-distance-lower} by the following linear one with universal constant, without assuming a uniform interior sphere condition.
	
	\begin{lemma}\label{lem:linear-boundary-growth}
		Let $\Omega$ be smooth and strictly convex, satisfy $|\Omega|=\pi$ and \eqref{eq:uniform-geometry}, and let $v$ be defined by \eqref{eq:v-def}--\eqref{eq:v-problem}. There exists a universal constant $a_*>0$ such that for all sufficiently small $\varepsilon$,
		\begin{equation}\label{eq:linear-boundary-growth}
			\varepsilon-v(x)
			\ge a_*\,d(x,\partial\Omega).
		\end{equation}
	\end{lemma}
	
	\begin{proof}
		Set
		\[
		w:=\varepsilon-v.
		\]
		Then
		\[
		w\geq 0 \quad \text{in }\Omega,
		\qquad
		-\Delta w=2 \quad \text{in }\Omega,
		\qquad
		\partial_\nu w=-c_\Omega \quad \text{on }\partial\Omega.
		\]
		\eqref{eq:global-gradient}   implies  
		\[
		\|\nabla w\|_{L^\infty(\Omega)}\leq L_*
		\]
		for a universal constant \(L_*>0\). Moreover, note that 
		\[
		P(\Omega)\leq \pi\operatorname{diam}(\Omega)\leq \pi D_*,
		\]
		we have
		\[
		c_\Omega=\frac{2\pi}{P(\Omega)}
		\geq \frac{2}{D_*}
		=:c_*>0.
		\]
		
		Suppose that the conclusion is false. Then there exist smooth convex
		domains \(\Omega_j\), corresponding functions \(w_j\), and points
		\(x_j\in\Omega_j\) such that
		\[
		\frac{w_j(x_j)}{d_j}\longrightarrow0, \qquad
		d_j:=d(x_j,\partial\Omega_j).
		\]
		If \(d_j\geq d_1>0\) along a subsequence, then convex compactness and
		the uniform Lipschitz bound yield a nonnegative local limit
		\(w_\infty\) satisfying
		\[
		-\Delta w_\infty=2
		\]
		and vanishing at an interior point, contradicting the strong maximum
		principle. Hence
		\[
		d_j\longrightarrow0.
		\]
		
		Let \(y_j\in\partial\Omega_j\) be a nearest boundary point to \(x_j\).
		After a rigid motion, assume that
		\[
		y_j=0,
		\qquad
		x_j=d_je_2,
		\qquad
		\Omega_j\subset\{x_2>0\}.
		\]
		Define
		\[
		K_j:=d_j^{-1}\Omega_j,
		\qquad
		W_j(X):=\frac{w_j(d_jX)}{d_j}.
		\]
		Then
		\[
		B_1(e_2)\subset K_j,
		\qquad
		0\in\partial K_j,
		\qquad
		K_j\subset\{X_2>0\},
		\]
		and
		\[
		W_j(e_2)=\frac{w_j(x_j)}{d_j}\longrightarrow0.
		\]
		Furthermore, 
		\[
		\|\nabla W_j\|_{L^\infty(K_j)}
		=
		\|\nabla w_j\|_{L^\infty(\Omega_j)}
		\leq L_*.
		\]
		Consequently, for   fixed \(R>0\),  one  finds 
		\[
		\sup_{K_j\cap B_R}|W_j|
		\leq
		|W_j(e_2)|+L_*(R+1),
		\]
		so \(W_j\) are locally uniformly bounded and equi-Lipschitz.
		
		In view of   \(e_2\in K_j\), the distance functions
		\[
		X\longmapsto\operatorname{dist}(X,\overline{K_j})
		\]
		are locally uniformly bounded and uniformly \(1\)-Lipschitz. Thus,
		after passing to a subsequence,
		\[
		\overline{K_j}\longrightarrow K
		\]
		locally in the Hausdorff topology, where \(K\) is a nonempty closed
		convex set satisfying
		\[
		B_1(e_2)\subset K\subset\{X_2\geq0\}.
		\]
		
		Each \(W_j\)  can be  extended to \(\mathbb{R}^2\) with the same Lipschitz constant, for instance by the McShane extension theorem.
		The local bounds
		above and the Arzel\`a--Ascoli theorem   ga guarantees that, after passing to a further subsequence,   these extensions converge locally uniformly to a
		Lipschitz function.  Denote its restriction to \(K\) by \(W\).
		
		The rescaled equation is
		\[
		-\Delta W_j=2d_j
		\quad\text{in }K_j.
		\]
		Let
		\[
		\phi\in C_c^\infty(K^\circ).
		\]
		Then	for   sufficiently large \(j\),  we obtain 
		\[
		\operatorname{supp}\phi\subset K_j.
		\]
		Therefore
		\[
		-\int_{K_j}W_j\Delta\phi\,dX
		=
		2d_j\int_{K_j}\phi\,dX.
		\]
		Passing to the limit   shows 
		\[
		-\int_{K^\circ}W\Delta\phi\,dX=0.
		\]
		Hence \(W\) is harmonic in \(K^\circ\) by Weyl's lemma.  In light of
		\[
		W\geq0,
		\qquad
		e_2\in K^\circ,
		\qquad
		W(e_2)=0,
		\]
		the strong maximum principle yields
		\[
		W\equiv0
		\quad\text{in }K^\circ.
		\]
		Since \(K\) is convex with nonempty interior and \(W\) is continuous,
		\[
		W\equiv0
		\quad\text{on }K.
		\]
		
		We claim that for   fixed \(R>0\),
		\[
		\sup_{K_j\cap B_R}W_j\longrightarrow0.
		\]
		Otherwise, there would exist \(\alpha>0\) and
		\(X_j\in K_j\cap B_R\) such that
		\[
		W_j(X_j)\geq\alpha.
		\]
		Then, after passing to a subsequence,
		\[
		X_j\longrightarrow X\in K\cap\overline{B_R},
		\]
		and the local uniform convergence of the extensions  show that 
		\[
		W(X)\geq\alpha,
		\]
		a contradiction.
		
		Let  
		\[
		\varphi\in C_c^\infty(\mathbb{R}^2),
		\qquad
		\operatorname{supp}\varphi\subset B_R.
		\]
		Integration by parts implies 
		\[
		\begin{aligned}
			\int_{K_j}\nabla W_j\cdot\nabla\varphi\,dX
			={}&
			-\int_{K_j}W_j\Delta\varphi\,dX
			\\
			&+
			\int_{\partial K_j}
			W_j\partial_{\nu_j}\varphi\,d\mathcal{H}^1.
		\end{aligned}
		\]
		The first term tends to zero. Moreover, convexity  entails  the standard
		local perimeter bound
		\[
		\mathcal{H}^1(\partial K_j\cap B_R)\leq C_R,
		\]
		where \(C_R\) is independent of \(j\). Hence the second term also
		tends to zero, and then
		\[
		\int_{K_j}\nabla W_j\cdot\nabla\varphi\,dX
		\longrightarrow0.
		\]
		
		On the other hand, the rescaled Neumann weak formulation is
		\[
		\int_{K_j}\nabla W_j\cdot\nabla\varphi\,dX
		=
		2d_j\int_{K_j}\varphi\,dX
		-
		c_{\Omega_j}
		\int_{\partial K_j}\varphi\,d\mathcal{H}^1.
		\]
		Fix \(r\in(0,1/4)\), and choose
		\[
		\varphi\in C_c^\infty(B_{2r}),
		\qquad
		0\leq\varphi\leq1,
		\qquad
		\varphi\equiv1 \quad\text{on }B_r.
		\]
		On account of
		\[
		0\in\partial K_j
		\qquad\text{and}\qquad
		B_1(e_2)\subset K_j,
		\]
		the boundary \(\partial K_j\) contains an arc starting at \(0\) and
		reaching \(\partial B_r\). Thus
		\[
		\mathcal{H}^1(\partial K_j\cap\overline{B_r})\geq r,
		\]
		and hence
		\[
		\int_{\partial K_j}\varphi\,d\mathcal{H}^1\geq r.
		\]
		It follows that
		\[
		\int_{K_j}\nabla W_j\cdot\nabla\varphi\,dX
		\leq
		2d_j|B_{2r}|-c_*r.
		\]
		So the right-hand side is bounded above by
		\[
		-\frac{c_*r}{2}
		\]
		for all sufficiently large \(j\), contradicting
		\[
		\int_{K_j}\nabla W_j\cdot\nabla\varphi\,dX
		\longrightarrow0.
		\]
		This ends  proof.
	\end{proof}
	
	Now we are in the position to prove the sharp stability inequality.
	
	\begin{proof}[Proof of Theorem~\ref{thm:intro-quantitative}]
		We first assume that \(\Omega\) is smooth and strictly convex. Let \(z\)
		be the minimum point given by Lemma~\ref{lem:minimum-depth}, and let
		\(q\) and \(h\) be defined by \eqref{eq:q-h-def}. Set
		\[
		\rho_i:=\min_{x\in\partial\Omega}|x-z|,
		\qquad
		\rho_e:=\max_{x\in\partial\Omega}|x-z|,
		\qquad
		\eta:=\rho_e-\rho_i.
		\]
		
		By \eqref{eq:radii-q} and Corollary~\ref{cor:oscillation-h},
		after decreasing the smallness threshold if necessary, we have
		\[
		\eta
		\le C\varepsilon^{1/2}\log\frac{C}{\varepsilon}
		\le \frac12.
		\]
		The fact
		$
		B_{\rho_i}(z)\subset\Omega\subset B_{\rho_e}(z) 
		$
		yields 
		\[
		\rho_i\leq (x-z)\cdot\nu(x)\leq\rho_e
		\qquad\text{on }\partial\Omega.
		\]
		Then the perimeter monotonicity  and 	 the area normalization implies
		\[
		\frac{1}{\rho_e}\leq c_\Omega\leq\frac{1}{\rho_i}\qquad\text{and} \qquad
		\rho_i\leq1\leq\rho_e.
		\]
		We further have 
		\[
		\bigl|(x-z)\cdot\nu(x)-1\bigr|\leq\eta,
		\qquad
		|c_\Omega-1|\leq2\eta,
		\]
		and so
		\begin{equation}\label{eq:hnu-eta}
			\|h_\nu\|_{L^\infty(\partial\Omega)}
			\leq3\eta.
		\end{equation}
		
		Analogously to the proof of  Lemma \ref{lem:weighted-hessian},	we replace \eqref{eq:h-nu-bound} by   \eqref{eq:hnu-eta}, which   shows that the three boundary terms in  the reverse-Serrin identity
		\eqref{eq:MMP-identity} are respectively bounded  by
		\[
		C\varepsilon\eta,
		\qquad
		C\varepsilon\eta,
		\qquad
		C\varepsilon^2.
		\]
		Hence
		\begin{equation}\label{eq:improved-weighted-energy}
			\int_\Omega
			(\varepsilon-v)|D^2h|^2\,dx
			\leq
			C\varepsilon(\eta+\varepsilon) 
		\end{equation}
		and then 
		\begin{equation}\label{eq:d-weighted-Hessian-sharp}
			\int_\Omega
			d(x,\partial\Omega)|D^2h|^2\,dx
			\leq
			C\varepsilon(\eta+\varepsilon) 
		\end{equation}
		owing to Lemma~\ref{lem:linear-boundary-growth}. 
		Furthermore, 	Lemma~\ref{lem:uniform-weighted-oscillation} and Lemma~\ref{lem:minimum-depth}   yield 
		\begin{equation}\label{eq:osc-h-sharp}
			\operatorname{osc}_{\overline{\Omega}}h
			\leq
			C\sqrt{\varepsilon(\eta+\varepsilon)}.
		\end{equation}
		
		Note that 
		\[
		\frac12(\rho_e^2-\rho_i^2)
		=
		\operatorname{osc}_{\partial\Omega}q
		\leq
		\operatorname{osc}_{\overline{\Omega}}h+\varepsilon
		\]
		and\(\rho_e+\rho_i\geq1\).  It follows from
		\eqref{eq:osc-h-sharp} that
		\[
		\eta
		\leq
		C\sqrt{\varepsilon(\eta+\varepsilon)}
		+
		2\varepsilon.
		\]
		Applying  Young's inequality, we obtain
		\begin{equation}\label{eq:eta-linear}
			\eta\leq C\varepsilon.
		\end{equation}
		
		By
		\[
		R_\Omega-r_\Omega
		\leq
		\rho_e-\rho_i
		=
		\eta, \qquad\text{and} \qquad \rho_i\leq1\leq\rho_e,
		\]
		we deduce 	
		\[
		d_H\bigl(\Omega,B_1(z)\bigr)
		\leq
		\max\{1-\rho_i,\rho_e-1\}
		\leq
		\eta.
		\]
		As
		\[
		\varepsilon=2O(\Omega),
		\] 
		we establish	 \eqref{eq:intro-sharp-estimate}  for smooth strictly convex
		domains.
		
		Let  \(\Omega\) be an arbitrary bounded convex domain with
		$
		|\Omega|=\pi.
		$
		Choose smooth strictly convex domains \(\Omega_j\)  normalized by
		$
		|\Omega_j|=\pi 
		$
		such that
		\[
		\Omega_j\longrightarrow\Omega
		\]
		in the Hausdorff topology. As recalled in
		Section~\ref{sec:quantitative-full-convex},
		\[
		O(\Omega_j)\longrightarrow O(\Omega),
		\qquad
		r_{\Omega_j}\longrightarrow r_\Omega,
		\qquad
		R_{\Omega_j}\longrightarrow R_\Omega.
		\]
		For all sufficiently large \(j\), the smooth estimate applies to
		\(\Omega_j\). Moreover, the functional
		\[
		K\longmapsto
		\inf_{y\in\mathbb{R}^2}
		d_H\bigl(K,B_1(y)\bigr)
		\]
		is \(1\)-Lipschitz with respect to the Hausdorff distance. This completes the proof of \eqref{eq:intro-sharp-estimate}.
		
		We now prove that the linear order in \eqref{eq:intro-sharp-estimate} cannot be improved. For $|t|$ small, let $\widetilde\Omega_t$ be the smooth strictly convex domain with radial function
		\begin{equation}\label{eq:sharp-family}
			r_t(\theta)=1+t\cos2\theta,
		\end{equation}
		and let $\Omega_t=\lambda_t\widetilde\Omega_t$, where $\lambda_t=1+O(t^2)$ is chosen so that $|\Omega_t|=\pi$.
		
		We know	these domains are centrally symmetric, and their both an optimal inscribed disk and an optimal containing disk may be centered at the symmetry center. 
		
		Indeed, if $B_r(a)\subset K$, then also $B_r(-a)\subset K$, and the convexity implies $B_r(0)\subset K$. If $K\subset B_R(a)$, then for $x\in K$, one has 
		\begin{equation*}
			|x-a|\le R\qquad\text{and} \qquad |x+a|\le R,
		\end{equation*} 
		and so 
		$|x|^2+|a|^2\le R^2.$  In particular, $K\subset B_R(0)$.

		We derive from \eqref{eq:sharp-family} that
		\begin{equation}\label{eq:sharp-radii-expansion}
			R_{\Omega_t}-r_{\Omega_t}=2|t|+o(|t|).
		\end{equation}
		It remains to compute the boundary oscillation. On the unit disk, the unperturbed solution is
		\[
		u_0(r)=\frac{1-r^2}{4}.
		\]
		The first variations of area and perimeter vanish for the normal velocity $f(\theta)=\cos2\theta$. Hence, the compatible constant Neumann datum has zero first variation. Standard shape differentiability for the Neumann problem~\cite{DelfourZolesio2011,SokolowskiZolesio1992} provides 
		\[
		u_t=u_0+t u_1+o(t),
		\]
		where
		\[
		\Delta u_1=0\quad\text{in }B_1,
		\qquad
		\partial_r u_1=\frac12\cos2\theta\quad\text{on }\partial B_1.
		\]
		Thus, up to an additive constant,
		\[
		u_1(r,\theta)=\frac14r^2\cos2\theta.
		\]
		On the moving boundary,  we 	evaluate 
		\[
		u_t(1+t\cos2\theta,\theta)
		=u_0(1)+t\left[u_1(1,\theta)+u_0'(1)\cos2\theta\right]+o(t)
		=u_0(1)-\frac t4\cos2\theta+o(t) 
		\]
		for uniformly in $\theta$. 
		Moreover,	the area-normalizing homothety changes this only at order $O(t^2)$. Consequently,
		\begin{equation}\label{eq:sharp-O-expansion}
			O(\Omega_t)=\frac{|t|}{2}+o(|t|).
		\end{equation}
		Both \eqref{eq:sharp-radii-expansion} and \eqref{eq:sharp-O-expansion}  indicate 
		\begin{equation}\label{eq:sharp-ratio}
			\frac{R_{\Omega_t}-r_{\Omega_t}}{O(\Omega_t)}\longrightarrow4.
		\end{equation}
		Therefore no estimate with a power strictly larger than $1$ can hold, even among smooth strictly convex centrally symmetric domains. This proves that the linear order in Theorem~\ref{thm:intro-quantitative} is sharp. 
	\end{proof}

	\section{A one-sided averaged boundary deficit}\label{sec:averaged-deficit}
	
	In this section, we  show that the  qualitative picture survives   for the following natural strictly
	weaker one-sided averaged deficit	 
	\begin{equation*} 
		A(\Omega)
		:=\frac1{P(\Omega)}\int_{\partial\Omega}u_\Omega\,d\sigma
		-\min_{\partial\Omega}u_\Omega.
	\end{equation*}
	This quantity is independent of the additive normalization of $u_\Omega$.
	Under the normalization $\int_{\partial\Omega}u_\Omega\,d\sigma=0$,
	one simply has $A(\Omega)=-\min_{\partial\Omega}u_\Omega$. Moreover,
	\begin{equation}\label{eq:A-below-O}
		0\le A(\Omega)\le O(\Omega),
		\qquad
		A(\lambda\Omega)=\lambda^2A(\Omega).
	\end{equation}
	The inequality in \eqref{eq:A-below-O} is not in the right direction for deriving
	stability directly from the preceding results. The main point for the proof of  Theorem \ref{thm:intro-average}
	is that, in the planar convex class, the averaged deficit nevertheless rules
	out long-thin degeneration and can then be    quantitatively upgraded to control the full boundary oscillation.

	To this end,  we begin with a geometric variant of Proposition~\ref{prop:gap}. A different test function is chosen, rather than $x(1-x)$. Such choice is important: its derivative
	vanishes at both diametral endpoints and supplies the additional factor of the
	area needed for an $L^1$ boundary deficit.
	
	\begin{lemma} \label{lem:quartic-gap}
		Let $K\subset\R^2$ be convex with $\diam(K)=1$, and assume that
		$(0,0)$ and $(1,0)$ are endpoints of a diameter. Set
		\[
		\Phi(x):=x^2(1-x)^2.
		\]
		Then there exist universal constants $a_1,c_1>0$ such that, whenever
		$|K|=a\le a_1$,
		\begin{equation}\label{eq:quartic-gap}
			\frac1a\int_K\Phi\,dxdy
			-\frac1{P(K)}\int_{\partial K}\Phi\,d\sigma
			\ge c_1a.
		\end{equation}
	\end{lemma}
	
	\begin{proof}
		As in Section~\ref{sec:diameter-compactness}, write
		\[
		K=\{(x,y):0<x<1,\ -q(x)<y<p(x)\},
		\]
		where $p,q\ge0$ are concave and vanish at $0$ and $1$, and first
		argue for smooth $p$ and $q$. Let $h=p+q$, $H=h/a$, and let $\beta$
		be the measure in \eqref{eq:tent-representation}. In particular,
		$\beta([0,1])=2$. After a direct computation,  we obtain 
		\[
		\int_0^1\Phi(x)T_t(x)\,dx
		=\frac{1+t(1-t)\bigl(1+2t(1-t)\bigr)}{60}.
		\]
		Consequently, if $I_4:=\int_K\Phi$, then
		\begin{equation}\label{eq:quartic-bulk}
			\frac{I_4}{a}
			=\frac1{30}+\eta_4,
			\qquad
			\eta_4
			=\frac1{60}\int_0^1t(1-t)
			\bigl(1+2t(1-t)\bigr)\,d\beta(t)\ge0.
		\end{equation}
		
		Let $B_4:=\int_{\partial K}\Phi\,d\sigma$. Since
		\eqref{eq:weighted-length-excess} holds for both graphs  and
		$$2\int_0^1\Phi=1/15,\quad \Phi\le x(1-x)/4,$$ 
		there is an
		absolute constant $C_4$ such that
		\begin{equation}\label{eq:quartic-boundary}
			0\le B_4-\frac1{15}\le C_4a^2.
		\end{equation}
		Fix $M>2C_4+1$. If $\eta_4\ge Ma$, then $P(K)\ge2$ and
		\eqref{eq:quartic-bulk}--\eqref{eq:quartic-boundary} yield
		\eqref{eq:quartic-gap} for sufficiently small $a$.
		
		Suppose instead that $\eta_4<Ma$. The inequality 
		$1+2t(1-t)\ge1$  shows 
		\[
		\int_0^1t(1-t)\,d\beta(t)\le60Ma.
		\]
		We further deduce from \eqref{eq:psi-H-identity} that 
		\[
		\frac1a\int_Kx(1-x)\,dxdy-\frac16\le5Ma.
		\]
		With $5M$ in place of $M$, 	Lemma~\ref{lem:endpoint}  implies
		$P(K)\ge2+c_Ma$, From  which  and
		\eqref{eq:quartic-boundary}, we conclude, after decreasing $a_1$,
		\[
		\frac{B_4}{P(K)}\le\frac1{30}-c'_Ma.
		\]
		This proves \eqref{eq:quartic-gap}.
		The general convex case follows by the same smooth convex approximation
		used in Proposition~\ref{prop:gap}.
	\end{proof}
	
	\begin{lemma}
		\label{lem:A-thin-lower}
		There exist universal constants $a_A,c_A>0$ such that every bounded
		convex domain $K\subset\R^2$ with
		\[
		\diam(K)=1,
		\qquad |K|=a\le a_A,
		\]
		satisfies
		\begin{equation}\label{eq:A-thin-lower}
			A(K)\ge c_Aa.
		\end{equation}
	\end{lemma}
	
	\begin{proof}
		Choose diametral endpoints $(0,0)$ and $(1,0)$ and use the graph
		representation above. Let
		\[
		m:=\min_{\partial K}u_K,
		\qquad w:=u_K-m.
		\]
		By the maximum principle for $-\Delta w=1$, one has
		$w\ge0$ on $\overline K$, and
		\begin{equation}\label{eq:A-boundary-mass}
			\int_{\partial K}w\,d\sigma=P(K)A(K).
		\end{equation}
		Let $\tau$ be the Dirichlet torsion function of $K$.  Observe that 
		$h:=w-\tau$ is harmonic and nonnegative, the sharp multidimensional
		Hermite--Hadamard inequality 
		\cite{Larson2022}   guarantees 
		\[
		\frac1a\int_Kh\,dx
		\le\frac2{P(K)}\int_{\partial K}h\,d\sigma
		=2A(K).
		\]
		As in \eqref{eq:pqmax}, $K$ lies in a strip of width $Ca$.
		Comparison with the torsion function of this strip yields
		$0\le\tau\le Ca^2$, and hence
		\begin{equation}\label{eq:w-L1-A}
			\int_Kw\,dx\le2aA(K)+Ca^3.
		\end{equation}
		
		We test the weak formulation with $\Phi(x)=x^2(1-x)^2$. By
		Lemma~\ref{lem:quartic-gap},
		\begin{equation}\label{eq:A-test-lower}
			\int_K\nabla w\cdot\nabla\Phi\,dx
			=\int_K\Phi\,dx-
			\frac a{P(K)}\int_{\partial K}\Phi\,d\sigma
			\ge c_1a^2.
		\end{equation}
		We next explain the reason for the quartic choice. At almost every point
		of the upper graph,
		\[
		|\partial_\nu\Phi|
		=|\Phi'(x)|\frac{|p'(x)|}{\sqrt{1+p'(x)^2}}.
		\]
		If $p'(x)\ge0$, then concavity  implies $p'(x)\le p(x)/x\le2a/x.$
		On the other hand, if $p'(x)\le0$,  then one has
		$$|p'(x)|\le p(x)/(1-x)\le2a/(1-x).$$
		In both cases,  it follows from $|\Phi'(x)|\le2x(1-x)$ that 
		$|\partial_\nu\Phi|\le4a$. The lower graph is identical, and thus
		\begin{equation}\label{eq:quartic-normal-small}
			\|\partial_\nu\Phi\|_{L^\infty(\partial K)}\le4a.
		\end{equation}
		Integrating by parts and using 
		\eqref{eq:A-boundary-mass},   \eqref{eq:w-L1-A}, and $P(K)\le\pi\diam(K)=\pi$, one finds 
		\[
		\left|\int_K\nabla w\cdot\nabla\Phi\,dx\right|
		\le CaP(K)A(K)+C\int_Kw\,dx
		\le CaA(K)+Ca^3.
		\]
		Together with \eqref{eq:A-test-lower}, this gives
		$c_1a^2\le CaA(K)+Ca^3$, which ends the proof.
	\end{proof}
	
	To prove Theorem \ref{thm:intro-average}, we need a new ingredient: the boundary $P$ function estimate, which derives the following upper bound estimate for $O(\Omega)$ in terms of $A(\Omega)$.
	
	\begin{lemma}\label{lem:A-controls-O}
		Every bounded convex domain $\Omega\subset\mathbb R^2$ satisfies
		\begin{equation}\label{eq:A-controls-O}
			O(\Omega)
			\le
			\left(
			\frac{3P(\Omega)A(\Omega)}{8}
			\right)^{2/3}.
		\end{equation}
	\end{lemma}
	
	\begin{proof}
		We first assume that $\Omega$ is smooth and strictly convex. Write
		\[
		M:=\max_{\partial\Omega}u_\Omega,\qquad
		m:=\min_{\partial\Omega}u_\Omega,\qquad
		O:=M-m,
		\]
		and set $f:=M-u_\Omega$ on $\partial\Omega$. Repeating the
		$P$-function argument in the proof of Lemma~\ref{lem:boundary-gradient}
		(which does not require a smallness assumption for this boundary estimate)
		with $v=2(M-u_\Omega)$  establishes
		\[
		|\partial_\tau f|^2\le f
		\qquad\text{on }\partial\Omega.
		\]
		Hence $\sqrt f$ is $1/2$-Lipschitz with respect to arclength.
		
		Let $s=0$ be a point where $f=0$, and parametrize
		$\partial\Omega$ by arclength over $[-P/2,P/2]$, with the endpoints
		identified. Since $f$ attains the value $O$ somewhere on the boundary,
		the Lipschitz estimate yields 
		\[
		\sqrt O\le \frac{P}{4}.
		\]
		In particular, $[-2\sqrt O,2\sqrt O]\subset[-P/2,P/2]$.
		Moreover, the $1/2$-Lipschitz property of $\sqrt{f}$ also yields
		\[
		f(s)\le\frac{s^2}{4}.
		\]
		The equation $u_\Omega-m=O-f$  provides  
		\[
		\begin{aligned}
			P(\Omega)A(\Omega)
			&=\int_{\partial\Omega}(u_\Omega-m)\,d\sigma\\
			&\ge
			\int_{-2\sqrt O}^{2\sqrt O}
			\left(O-\frac{s^2}{4}\right)\,ds\\
			&=\frac83 O^{3/2}.
		\end{aligned}
		\]
		This proves
		\[
		O(\Omega)
		\le
		\left(
		\frac{3P(\Omega)A(\Omega)}8
		\right)^{2/3}.
		\]
		
		The general convex case follows by smooth strictly convex
		approximation and the moving-domain continuity of
		$O(\Omega)$, $A(\Omega)$, and $P(\Omega)$.
	\end{proof}
	
	Finally, we are ready to prove Theorem~\ref{thm:intro-average}.
	
	\begin{proof}[Proof of Theorem~\ref{thm:intro-average}]
		We first show that small $A$ guarantees  a universal geometric class. Suppose
		$|\Omega|=\pi$, let $L=\diam(\Omega)$, and set $K=L^{-1}\Omega$.
		Then
		\[
		\diam(K)=1,
		\qquad |K|=\frac\pi{L^2},
		\qquad A(K)=\frac{A(\Omega)}{L^2}.
		\]
		If $L$ is so large that $|K|\le a_A$, Lemma~\ref{lem:A-thin-lower}
		implies
		\[
		\frac{A(\Omega)}{L^2}=A(K)
		\ge c_A\frac\pi{L^2},
		\]
		or $A(\Omega)\ge c_A\pi$. Hence there exist universal constants
		$\delta_*>0$, $D_*>0$, and $r_*>0$ such that
		\begin{equation}\label{eq:A-uniform-geometry}
			A(\Omega)\le\delta_*
			\quad\Longrightarrow\quad
			\diam(\Omega)\le D_*,
			\qquad r_\Omega\ge r_*.
		\end{equation}
		Here the inradius bound follows from \eqref{eq:r-upbad}.
		
		In particular, $P(\Omega)\le\pi D_*$. Lemma~\ref{lem:A-controls-O}
		therefore  yields, after decreasing $\delta_*$,
		\begin{equation}\label{eq:O-from-A-uniform}
			O(\Omega)\le C A(\Omega)^{2/3}.
		\end{equation}
		\eqref{eq:O-from-A-uniform}  implies that,	if $A(\Omega_k)\to0$, then   
		$O(\Omega_k)\to0$, and the qualitative conclusion follows directly
		from Theorem~\ref{thm:disk-conv}.
		
		For the quantitative statement, choose $\delta_A\le\delta_*$ so small that the right-hand side of \eqref{eq:O-from-A-uniform} lies below the threshold in Theorem~\ref{thm:intro-quantitative}. Applying the sharp estimate \eqref{eq:intro-sharp-estimate} and  \eqref{eq:O-from-A-uniform}, we get \eqref{eq:2/3stability}.
		
	\end{proof}
	
	\begin{remark}\label{rem:A-scale-invariant}
		For a planar convex domain of arbitrary area, the natural dimensionless
		deficit is $\mathcal A(\Omega):=A(\Omega)/|\Omega|$. By scaling,
		Theorem~\ref{thm:intro-average} yields, whenever $\mathcal A(\Omega)$ is
		sufficiently small,
		\[
		\frac{R_\Omega-r_\Omega}{|\Omega|^{1/2}}
		+\frac1{|\Omega|^{1/2}}
		\inf_{z\in \mathbb{R}^2} d_H\bigl(\Omega,B_{\sqrt{|\Omega|/\pi}}(z)\bigr)
		\le C\mathcal A(\Omega)^{2/3}.
		\]
	\end{remark}
	
	\begin{remark}
		The exponent \(2/3\) in \eqref{eq:2/3stability} is sharp. For
		\(0<t<1/4\), let
		\[
		K_t=B_1\cap\{x_1<1-t\}.
		\]
		Geometrically, \(R_{K_t}-r_{K_t}\geq \frac{t}{2}\).
		We claim that for every \(\theta\in(0,1)\),
		$$
		A(K_t)\leq C_\theta
		t^{\frac{3\theta}{4-2\theta}}.
		$$
		Let \(\Omega_t=\lambda_tK_t\) be normalized so that
		\(\lvert\Omega_t\rvert=\pi\). Since
		\(\lambda_t=1+O(t^{3/2})\), we still have
		\[
		R_{\Omega_t}-r_{\Omega_t}\gtrsim t \quad\text{and}\quad A(\Omega_{t})\lesssim t^{\frac{3\theta}{4-2\theta}}
		\]
		Letting \(\theta\uparrow1\) in the claim shows that the exponent in
		\eqref{eq:2/3stability} cannot exceed \(2/3\).
		
		It remains to prove the claim. Set
		\[
		u_0=\frac{1-\lvert x\rvert^2}{4}
		\]
		and write the Neumann torsion function of \(K_t\) as
		\(u_t=u_0+h_t\), where \(h_t\) is harmonic. Since \(u_0=0\) on the
		circular part of \(\partial K_t\),
		\begin{align}
			\label{eq:2/3-0}
			A(K_t)
			\lesssim
			\int_{\partial K_t}u_0\,d\sigma
			+\osc_{\partial K_t}h_t \lesssim
			t^{3/2}+\osc_{\partial K_t}h_t.
		\end{align}
		
		Fix \(p_0>2\). Since the domains \(K_t\) have uniformly bounded
		Lipschitz character, \cite[Theorem~1.2]{GengShen} yields
		\(
		\|\nabla h_t\|_{L^{p_{0}}(K_t)}
		\) is uniformly bounded. Taking
		\[
		c=\min_{\overline{K_t}}h_t
		\]
		and applying H\"older's inequality, we obtain
		\begin{align}
			\label{eq:2/3-1}
			\|\nabla h_t\|_{L^2(K_t)}^2
			&=
			\int_{\partial K_t}
			(h_t-c)\,\partial_\nu h_t\,d\sigma \notag\\
			&\lesssim
			t^{3/2}\osc_{\partial K_t}h_t.
		\end{align}
		
		For \(p\in(2,p_0)\), choose \(\theta\in(0,1)\) such that
		\[
		\frac1p=\frac{\theta}{2}+\frac{1-\theta}{p_0}.
		\]
		Morrey's inequality and interpolation imply
		\begin{align}
			\label{eq:2/3-2}
			\osc_{K_t}h_t
			&\lesssim
			\|\nabla h_t\|_{L^p(K_t)} \notag\\
			&\lesssim
			\|\nabla h_t\|_{L^2(K_t)}^\theta
			\|\nabla h_t\|_{L^{p_0}(K_t)}^{1-\theta} \notag\\
			&\lesssim
			\|\nabla h_t\|_{L^2(K_t)}^\theta.
		\end{align}
		Combining \eqref{eq:2/3-1} and \eqref{eq:2/3-2}, we find
		\[
		\osc_{K_t}h_t
		\lesssim
		\bigl(t^{3/2}\osc_{K_t}h_t\bigr)^{\theta/2},
		\]
		and hence
		\[
		\osc_{K_t}h_t
		\lesssim
		t^{\frac{3\theta}{4-2\theta}}.
		\]
		The claim now follows from \eqref{eq:2/3-0}.
	\end{remark}

	\section*{Declaration of generative AI and AI-assisted technologies in the manuscript preparation process}
	
	During the preparation of this work, the authors used ChatGPT (OpenAI) to assist with language editing, manuscript organization, and exploratory discussions concerning the presentation of some	mathematical arguments. All mathematical ideas, statements, proofs,
	and references in the final manuscript were critically assessed, reworked where necessary, and independently verified by the authors. The authors reviewed and edited all AI-assisted output and take full responsibility for the content of the article.

\end{document}